\documentclass{article}
\usepackage[T2A]{fontenc}
\usepackage[cp1251]{inputenc}
\usepackage[english,russian]{babel}
\usepackage[tbtags]{amsmath}
\usepackage{amsfonts,amssymb,mathrsfs,amscd}

\usepackage{amsbib}

\usepackage[hyper]{msb-a}
\JournalName{}
\numberwithin{equation}{section}

\usepackage{mathrsfs}

\theoremstyle{plain}
\newtheorem{theorem}{Theorem}
\newtheorem{lemma}{Lemma}
\newtheorem{corollary}{Corollary}

\theoremstyle{definition}
\newtheorem{proof}{Proof}
\newtheorem{example}{Example}

\numberwithin{equation}{section} \allowdisplaybreaks

\newcounter{aa}

\let\mathcal=\mathscr

\begin{document}

\author[Anton\,S.~Galaev]{A\,S.~Galaev}

\udk{}

\title{Holonomy groups of Lorentzian manifolds}

\maketitle \markright{Holonomy groups of Lorentzian manifolds}

\begin{fulltext}

\begin{abstract}
In this paper, a survey of the recent results about the
classification of the connected holonomy groups of the Lorentzian
manifolds is given. A simplification of the construction of the
Lorentzian metrics with all possible connected holonomy groups is
obtained. As the applications, the Einstein equation, Lorentzian
manifolds with parallel and recurrent spinor fields, conformally
flat Walker metrics and the classification of 2-symmetric
Lorentzian manifolds are considered.

Bibliography: 123 titles.
\end{abstract}


\setcounter{tocdepth}{2} \tableofcontents


\section{Introduction}
\label{sec1}

The notion of the holonomy group  was introduced for the first
time in the works of \'E.~Cartan~\cite{42} and~\cite{44},
in~\cite{43} he used the holonomy groups in order to obtain the
classification of the Riemannian symmetric spaces. The holonomy
group of a pseudo-Riemannian  manifold is the Lie subgroup of the
Lie group of pseudo-orthogonal transformations of the tangent
space at a point of the manifold and it consists of parallel
transports along piece-wise smooth loops at this point. Usually
one considers the connected holonomy group, i.e.,  the connected
component of the identity of the holonomy group, for its
definition it is necessary to consider parallel transports along
contractible loops. The Lie algebra corresponding to the holonomy
group is called the holonomy algebra. The holonomy group of a
pseudo-Riemannian manifold is an invariant of the corresponding
Levi-Civita connection; it gives information about the curvature
tensor and about parallel sections of the vector bundles
associated to the manifold,  such as the tensor bundle or the
spinor bundle.

An important result is the Berger classification of the connected
irreducible holonomy groups of Riemannian manifolds~\cite{23}. It
turns out that the connected holonomy group of an~$n$-dimensional
indecomposable not locally symmetric Riemannian manifold is
contained in the following list: $\operatorname{SO}(n)$;
$\operatorname{U}(m)$, $\operatorname{SU}(m)$ ($n=2m$);
$\operatorname{Sp}(m)$, $\operatorname{Sp}(m)\cdot
\operatorname{Sp}(1)$ ($n=4m$); $\operatorname{Spin}(7)$ ($n=8$);
$G_2$~($n=7$). Berger obtained merely a list of possible holonomy
groups, and the problem to show that there exists a manifold with
each of these holonomy groups arose. In particular, this resulted
to the famous Calabi-Yau Theorem~\cite{123}. Only in 1987
Bryant~\cite{38} constructed examples of Riemannian manifolds with
the holonomy groups~$\operatorname{Spin}(7)$ and~$G_2$. Thus the
solution of this problem required more then thirty years. The
de~Rham decomposition Theorem~\cite{48} reduces the classification
problem for the connected  holonomy groups of Riemannian manifolds
to the case of the irreducible  holonomy groups.

Indecomposable Riemannian manifolds with special (i.е., different
from $\operatorname{SO}(n)$)  holonomy groups have important
geometric properties. Manifolds with the most of these holonomy
groups are Einstein or Ricci-flat and admit parallel spinor
fields. These properties ensured that  the Riemannian manifolds
with special holonomy groups found applications in theoretical
physics (in string theory, supersymmetry theory and
M-theory)~\cite{25},~\cite{45},~\cite{79},
\cite{88},~\cite{89},~\cite{103}. In this connection during the
last 20 years appeared a great number of works, where
constructions of complete and compact  Riemannian manifolds with
special holonomy groups are described, let us cite only some of
these works:~\cite{18},~\cite{20},~\cite{47},~\cite{51},
\cite{88},~\cite{89}. It is important to note that in the string
theory and M-theory it is assumed that our space is locally a
product
\begin{equation}
\label{eq1.1} \mathbb{R}^{1,3}\times M
\end{equation}
of the Minkowski apace~$\mathbb{R}^{1,3}$ and of some compact
Riemannian manifold~$M$ of dimension~$6$,~$7$ or~$8$ and with the
holonomy group $\operatorname{SU}(3)$, $G_2$ or
$\operatorname{Spin}(7)$, respectively. Parallel spinor fields
on~$M$ define supersymmetries.

It is natural to consider the classification problem of connected
holonomy groups of pseudo-Riemannian manifolds, and first of all
of Lorentzian manifolds.

There is the Berger classification of connected irreducible
holonomy groups of pseudo-Riemannian manifolds~\cite{23}. However,
in the case of  pseudo-Riemannian manifolds it is not enough to
consider only irreducible holonomy groups. The Wu decomposition
Theorem~\cite{122} allows to restrict the consideration to the
connected weakly irreducible holonomy groups. A weakly irreducible
 holonomy group does not preserve any nondegenerate proper vector subspace of
 the tangent space. Such holonomy group may preserve degenerate subspace of the tangent space.
 In this case the holonomy group is not reductive.   Therein lies the main problem.

A long time there were solely results about the holonomy groups of
four-dimensional Lorentzian
manifolds~\cite{10},~\cite{81},~\cite{87},
~\cite{91},~\cite{92},~\cite{102},~\cite{109}. In these works the
classification of the connected  holonomy groups is obtained, the
relation with the Einstein equation, the Petrov classification of
the gravitational fields~\cite{108} and with other problems of
General relativity is considered.

In~1993, B\'erard-Bergery and Ikemakhen made the first step
towards the classification of the connected holonomy groups for
Lorentzian manifolds of arbitrary dimension~\cite{21}. We describe
all subsequent steps of the classification and its consequences.

In Section~\ref{sec2} of the present paper, definitions and some
known results about the holonomy groups of Riemannian and
pseudo-Riemannian manifolds are set out.

In Section~\ref{sec3} we start to study the holonomy algebras
$\mathfrak{g}\subset\mathfrak{so}(1,n+1)$ of Lorentzian
manifolds~$(M,g)$ of dimension~$n+2\geqslant 4$. The Wu Theorem
allows to assume that the holonomy algebra is weakly irreducible.
If $\mathfrak{g}\ne\mathfrak{so}(1,n+1)$, then~$\mathfrak{g}$
preserves an isotropic line of the tangent space and it is
contained in the maximal
subalgebra~$\mathfrak{sim}(n)\subset\mathfrak{so}(1,n+1)$
preserving this line. First of all we give a geometric
interpretation~\cite{57} of the classification by B\'erard-Bergery
and Ikemakhen~\cite{21} of weakly irreducible subalgebras
in~$\mathfrak{g}\subset\mathfrak{sim}(n)$. It turns out that these
algebras are exhausted by the Lie algebras of transitive groups of
similarity transformations of the Euclidean space~$\mathbb{R}^n$.

Next we study the question, which of the obtained subalgebras
$\mathfrak{g}\subset\mathfrak{sim}(n)$ are the holonomy algebras
of Lorentzian manifolds. First of all, it is necessary to classify
the Berger subalgebras  $\mathfrak{g}\subset\mathfrak{sim}(n)$,
these algebras are spanned by the images of the elements of the
space~$\mathscr{R}(\mathfrak{g})$ of the algebraic curvature tensors
(tensors, satisfying the first Bianchi identity) and they are
candidates to the holonomy algebras. In~Section~\ref{sec4} we
describe the structure of the spaces of curvature
tensors~$\mathscr{R}(\mathfrak{g})$ for the  subalgebras
$\mathfrak{g}\subset\mathfrak{sim}(n)$~\cite{55} and reduce the
classification problem for the Berger algebras to the
classification problem for the weak Berger algebras
$\mathfrak{h}\subset\mathfrak{so}(n)$, these algebras are spanned
by the images of the elements of the
space~$\mathscr{P}(\mathfrak{h})$, consisting of the linear maps
from~$\mathbb{R}^n$ to~$\mathfrak{h}$ and satisfying some
identity. Next we find the  curvature tensor of the Walker
manifolds, i.e., manifolds with the holonomy algebras
$\mathfrak{g}\subset\mathfrak{sim}(n)$.

In~Section~\ref{sec5} the results of computations of the
spaces~$\mathscr{P}(\mathfrak{h})$ from~\cite{59} are given. This
gives the complete structure of the spaces of  curvature tensors
for the holonomy algebras $\mathfrak{g}\subset\mathfrak{sim}(n)$.
The space~$\mathscr{P}(\mathfrak{h})$ appeared as the space of
values of a component of the curvature tensor of a Lorentzian
manifold. Later it turned out that to this space belongs also a
component of the  curvature tensor of a Riemannian
supermanifold~\cite{63}.

Leistner~\cite{100} classified weak Berger algebras, showing in a
far non-trivial way that they are exhausted by the  holonomy
algebras of  Riemannian spaces. The natural problem  to get a
direct simple proof of this fact arises. In~Section~\ref{sec6} we
give such a proof from~\cite{68} for the case of semisimple not
simple irreducible Lie
algebras~$\mathfrak{h}\subset\mathfrak{so}(n)$. The Leistner
Theorem implies the classification of the Berger
subalgebras~$\mathfrak{g}\subset\mathfrak{sim}(n)$.

In~Section~\ref{sec7} we prove that all  Berger algebras may be
realized as the holonomy algebras of Lorentzian manifolds, we
greatly simplify the constructions of the metrics from~\cite{56}.
By this we complete the classification of the holonomy algebras of
Lorentzian manifolds.

The problem to construct examples of Lorentzian manifolds with
various  holonomy groups and additional global geometric
properties springs up. In~\cite{17},~\cite{19} constructions of
globally hyperbolic Lorentzian manifolds  with some classes of the
holonomy groups are given. The global hyperbolicity is a strong
casuality condition in Lorentzian geometry that generalizes the
general notion of  completeness in Riemannian geometry.
In~\cite{95} some constructions using the Kaluza-Klein idea are
suggested. In the papers~\cite{16},~\cite{97},~\cite{101} various
global geometric properties of  Lorentzian manifolds with
different holonomy groups are studied. The holonomy groups are
discussed in the recent survey on global  Lorentzian
geometry~\cite{105}. In~\cite{16} Lorentzian manifolds with
disconnected holonomy groups are considered, some examples are
given. In~\cite{70},~\cite{71} we give algorithms allowing to
compute the holonomy algebra of an arbitrary Lorentzian manifold.

Next we consider some applications of the obtained classification.

In~Section~\ref{sec8} we study the relation of the  holonomy
algebras and the Einstein equation. The subject is motivated by
the paper by the theoretical physicists Gibbons and
Pope~\cite{76}, in which the problem of finding the the Einstein
metrics with the holonomy algebras in~$\mathfrak{sim}(n)$ was
proposed, examples were considered and their physical
interpretation was given. We find the  holonomy algebras of the
Einstein Lorentzian manifolds~\cite{60},~\cite{61}. Next we show
that on each Walker manifold there exist special coordinates
allowing to simplify appreciably the Einstein equation~\cite{74}.
Examples of Einstein metrics from~\cite{60},~\cite{62} are given.

In~Section~\ref{sec9} results about Riemannian and Lorentzian
manifolds admitting recurrent spinor fields~\cite{67} are
presented. Recurrent spinor fields generalize parallel spinor
fields. Simply connected Riemannian manifolds with parallel spinor
fields were classified in~\cite{121} in terms of their holonomy
groups. Similar problem for  Lorentzian manifolds was considered
in~\cite{40},~\cite{52}, and it was solved
in~\cite{98},~\cite{99}. The relation of the  holonomy groups of
Lorentzian manifolds with the solutions of some other spinor
equations is discussed in~\cite{12},~\cite{13},~\cite{17} and in
physical literature that is cited below.

In~Section~\ref{sec10} the local classification of  conformally
flat Lorentzian manifolds with special holonomy groups~\cite{66}
is obtained.  The
corresponding local metrics  are certain extensions of Riemannian
spaces of constant sectional curvature to Walker metrics. It is
noted that earlier there was a problem to find examples of such
metrics in dimension~4~\cite{75},~\cite{81}.

In~Section~\ref{sec11} we obtain the classification of 2-symmetric
Lorentzian manifolds, i.e., manifold satisfying the
condition~$\nabla^2 R=0$, $\nabla R\ne\nobreak0$. We discuss and
simplify the proof of this result from~\cite{5}, demonstrating the
applications of the holonomy groups theory. The classification
problem for 2-symmetric manifolds was studied also
in~\cite{28},~\cite{29},~\cite{90},~\cite{112}.

Lorentzian manifolds with weakly irreducible not irreducible
holonomy groups admit parallel distributions of isotropic lines;
such manifolds are also called the Walker
manifolds~\cite{37},~\cite{120}. These manifolds are studied  in
geometric and physical literature.
In~works~\cite{35},~\cite{36},~\cite{77} the hope is expressed
that the Lorentzian manifolds with special  holonomy groups will
find applications in theoretical physics, e.g., in M-theory and
string theory. It is suggested to replace the
manifold~\eqref{eq1.1} by an indecomposable  Lorentzian manifold
with an appropriate
 holonomy group. Recently in connection with the 11-dimensional supergravity theory  appeared physical
 works, where  11-dimensional Lorentzian manifolds admitting spinor fields satisfying some
 equation are studied. At that the holonomy groups are used~\cite{11},~\cite{53},~\cite{113}.
Let us mention also the works~\cite{45},~\cite{46},~\cite{78}. All
that shows the importance of the study of the  holonomy groups of
Lorentzian manifolds and the related geometric structures.

I the case of  pseudo-Riemannian manifold of signatures different
from the Riemannian and Lorentzian ones the classification of the
holonomy groups is absent. There are some partial results
only~\cite{22},~\cite{26},~\cite{27},
~\cite{30},~\cite{58},~\cite{65},~\cite{69},~\cite{73},~\cite{85}.

Finally let us mention some other results about holonomy groups.
The consideration of the cone over a Riemannian manifold allows to
obtain  Riemannian metrics with special holonomy groups and
interpret the Killing spinor fields as the parallel spinor fields
on the cone~\cite{34}. To that in the paper~\cite{4} the holonomy
groups of the cones over pseudo-Riemannian manifolds, and in
particular over Lorentzian manifolds, are studied. There are
results about irreducible holonomy groups of linear torsion-free
connections~\cite{9},~\cite{39}, \cite{104},~\cite{111}. The
holonomy groups are defined also for manifolds with conformal
metrics,in particular, these groups allow to decide if there are
Einstein metrics in the conformal class~\cite{14}. The notion of
the  holonomy group is used also for connections on
supermanifolds~\cite{1},~\cite{63}.

The author is thankful to D.\,V~Alekseevsky for useful discussions
and suggestions.

\section{Holonomy groups and algebras: definitions and facts}
\label{sec2}

In this section we recall some definitions and known facts about
holonomy groups of pseudo-Riemannian
manifolds~\cite{25},~\cite{88}, ~\cite{89},~\cite{94}. All
manifolds are assumed to be connected.

\subsection[{Holonomy groups of connections in vector bundles}]{Holonomy groups of connections in vector bundles}
\label{ssec2.1} Let $M$ be a smooth  manifold and~$E$ be a vector
bundle over~$M$ with a connection~$\nabla$. The connection defines
the parallel transport: for any piece-wise smooth curve
$\gamma\colon[a,b]\subset\mathbb{R}\to M$ an isomorphism
$$
\tau_{\gamma}\colon E_{\gamma(a)}\to E_{\gamma(b)}
$$ of the vector spaces is defined.
Let us fix a point $x\in M$. The holonomy group~$G_x$ of the
connection~$\nabla$ at the point~$x$ is the group consisting of
parallel transports along all piecewise smooth loops at the
point~$x$. If we consider only null-homotopic loops, we get  the
restricted holonomy group $G^0_x$.  If the manifold~$M$ is simply
connected, then $G^0_x=G_x$. It is known that the group $G_x$ is a
Lie subgroup of the Lie group~$\operatorname{GL}(E_x)$ and the
group~$G^0_x$ is the connected identity component of the Lie
group~$G_x$. Let~$\mathfrak{g}_x\subset\mathfrak{gl}(E_x)$ be the
corresponding Lie algebra; this algebra is called the holonomy
algebra of the connection~$\nabla$ at the point~$x$. The holonomy
groups  at different points of a  connected manifold are
isomorphic, and one can speak about the holonomy group
$G\subset\operatorname{GL}(m,\mathbb{R})$, or about the holonomy
algebra $\mathfrak{g}\subset\mathfrak{gl}(m,\mathbb{R})$ of the
connection~$\nabla$ (here $m$ is the rank of the vector
bundle~$E$). In the case of a simply connected manifold, the
holonomy algebra determines the holonomy group uniquely.

Recall that a section $X\in\Gamma(E)$ is called parallel if
$\nabla X=0$. This is equivalent to the condition that  for any
piece-wise smooth curve $\gamma:[a,b]\to M$ holds $\tau_\gamma
X_{\gamma(a)}=X_{\gamma(b)}$. Similarly, a subbundle $F\subset E$
is called parallel  if for any  section~$X$ of the subbundle~$F$
and for any vector field~$Y$ on~$M$, the section~$\nabla_YX$ again
belongs to~$F$. This is equivalent to the property, that for any
piece-wise smooth curve  $\gamma\colon[a,b]\to M$ it holds
$\tau_{\gamma}F_{\gamma(a)}=F_{\gamma(b)}$.

The importance of holonomy groups shows the following fundamental
principle.

\begin{theorem}
\label{th1} There exists a one-to-one correspondence between
parallel sections $X$ of the bundle $E$ and vectors $X_x\in E_x$
invariant with respect to~$G_x$.
\end{theorem}

Let us describe this correspondence. Having a parallel section~$X$
it is enough to take the value $X_x$ at the point~$x\in M$.
Since~$X$ is invariant under the parallel transports, the
vector~$X_x$ is invariant under   the holonomy group. Conversely,
for a given vector~$X_x$  define the section~$X$. For any point
$y\in M$ put $X_y=\tau_{\gamma}X_x$,
 where $\gamma$ is any curve beginning at  $x$ and ending at the point $y$. The value $X_y$ does not depend
 on the choice of the curve~$\gamma$.

A similar result holds for subbundles.

\begin{theorem}
\label{th2} There exists a one-to-one correspondence between
parallel subbundles $F\subset E$ and vector subspaces ${F_x\subset
E_x}$ invariant with respect to~$G_x$.
\end{theorem}

The next theorem proven by  Ambrose and Singer~\cite{8} shows the
relation of the holonomy algebra and the curvature tensor~$R$ of
the connection~$\nabla$.

\begin{theorem}
\label{th3} Let $x\in M$. The Lie algebra~$\mathfrak{g}_x$ is
spanned by the operators of the following form:
$$
\tau_{\gamma}^{-1}\circ
R_y(X,Y)\circ\tau_{\gamma}\in\mathfrak{gl}(E_x),
$$
where $\gamma$
 is an arbitrary piece-wise smooth
 curve beginning at the point~$x$ and anding at a point~$y\in M$, and $Y,Z\in T_yM$.
\end{theorem}

\subsection{Holonomy groups of pseudo-Riemannian manifolds}
\label{ssec2.2} Let us consider pseudo-Riemannian manifolds.
Recall that a pseudo-Riemannian manifold of signature $(r,s)$ is a
smooth manifold~ $M$ equipped with a smooth field~$g$ of symmetric
non-degenerate bilinear forms of signature $(r,s)$ ($r$ is the
number of minuses) at each point.  If $r=0$, then such manifold is
called a Riemannian manifold. If $r=1$,  then~$(M,g)$ is a
Lorentzian manifold. In this case  for the contentious we assume
that $s=n+1$, $n\geqslant 0$.

On the tangent bundle~$TM$ of a  pseudo-Riemannian manifold~$M$
one canonically gets the  Levi-Civita connection~$\nabla$ defined
by the following two conditions: the field of forms~$g$ is
parallel $(\nabla g=0)$ and the torsion is zero
$(\operatorname{Tor}=0)$. Denote by~$\mathrm{O}(T_xM,g_x)$ the
group of linear transformation of the space~$T_xM$ preserving the
form~$g_x$. Since the metric~$g$ is parallel, $G_x\subset
\mathrm{O}(T_xM,g_x)$. The tangent space~$(T_xM,g_x)$ can be
identified with the pseudo-Euclidean space~$\mathbb{R}^{r,s}$, the
metric of this space we denote  by the symbol~$g$. Then we may
identify the  holonomy group~$G_x$ with a Lie subgroup
in~$\mathrm{O}(r,s)$, and the holonomy algebra~$\mathfrak{g}_x$
with a subalgebra in~$\mathfrak{so}(r,s)$.

The connection~$\nabla$ is in a natural way extendable to a
connection in the tensor bundle $\otimes^p_qTM$, the holonomy
group of this connection coincides with the natural representation
of the group~$G_x$ in the  tensor space~$\otimes^p_qT_xM$. The
following statement follows from Theorem~\ref{th1}.

\begin{theorem}
\label{th4} There exists a one-to-one correspondence between
parallel tensor fields~$A$ of type~$(p,q)$  and tensors
$A_x\in\otimes^p_qT_xM$  invariant with respect to~$G_x$.
\end{theorem}

Thus if we know the  holonomy group of a manifold, then the
geometric problem of finding the parallel  tensor fields on the
 manifold can be reduced to the more simple  algebraic problem of finding
the invariants of the  holonomy group. Let us consider several examples illustrating this
 principle.

Recall that a pseudo-Riemannian manifold~$(M,g)$ is called flat
if~$(M,g)$ admits local parallel fields of frames. We get
that~$(M,g)$ is flat if and only if  $G^0=\{\operatorname{id}\}$
(or $\mathfrak{g}=\{0\}$). Moreover, from the Ambrose-Singer
Theorem it follows that the last equality is equivalent to the
nullity of the  curvature tensor.

Next,  a pseudo-Riemannian manifold $(M,g)$ is called
\textit{pseudo-K\"ahlerian} if on~$M$ there exists a parallel
field of endomorphisms~$J$ with the properties
$J^2=-\operatorname{id}$ and $g(JX,Y)+g(X,JY)=0$ for all vector
fields~$X$ and~$Y$ on~$M$. It is obvious that a pseudo-Riemannian
manifold $(M,g)$ of signature  $(2r,2s)$ is pseudo-K\"ahlerian if
and only if $G\subset \operatorname{U}(r,s)$.

For an arbitrary subalgebra $\mathfrak{g}\subset
\mathfrak{so}(r,s)$ let
\begin{align*}
\mathscr{R}(\mathfrak{g})=\bigl\{R\in\operatorname{Hom}
(\wedge^2\mathbb{R}^{r,s},\mathfrak{g})\mid{}
&R(X,Y)Z+R(Y,Z)X+R(Z,X)Y=0
\\
&\text{for all} \ X,Y,Z\in \mathbb{R}^{r,s}\bigr\}.
\end{align*}
The space~$\mathscr{R}(\mathfrak{g})$ is called the space of
curvature tensors of type~$\mathfrak{g}$. We denote
by~$L(\mathscr{R}(\mathfrak{g}))$ the vector subspace
of~$\mathfrak{g}$ spanned by the elements of the form $R(X,Y)$ for
all $R\in\mathscr{R}(\mathfrak{g})$, $X,Y\in \mathbb{R}^{r,s}$.
From the Ambrose-Singer Theorem and the first Bianchi identity it
follows that if $\mathfrak{g}$ is the holonomy algebra of a
pseudo-Riemannian space $(M,g)$ at a point~$x\in M$, then $R_x\in
\mathscr{R}(\mathfrak{g})$, i.e., the knowledge of the  holonomy
algebra allows to get restrictions on the curvature tensor, this
will be used repeatedly below. Moreover, it holds
$L(\mathscr{R}(\mathfrak{g}))=\mathfrak{g}$. A subalgebra
$\mathfrak{g}\subset \mathfrak{so}(r,s)$ is called a Berger
algebra if the equality
$L(\mathscr{R}(\mathfrak{g}))=\mathfrak{g}$ is fulfilled. It is
natural to consider the
 Berger algebras  as the candidates to the holonomy algebras
of pseudo-Riemannian manifolds. Each element
$R\in\mathscr{R}(\mathfrak{so}(r,s))$ has the property
\begin{equation}
\label{eq2.1} (R(X,Y)Z,W)=(R(Z,W)X,Y),\qquad
X,Y,Z,W\in\mathbb{R}^{r,s}.
\end{equation}

Theorem~\ref{th3} does not give a good way to find the  holonomy
algebra. Sometimes it is possible to use the following theorem.

\begin{theorem}
\label{th5} If the pseudo-Riemannian manifold $(M,g)$ is analytic,
then the holonomy algebra~$\mathfrak{g}_x$ is generated by the
following operators:
$$
R(X,Y)_x,\nabla_{Z_1} R(X,Y)_x,\nabla_{Z_2}\nabla_{Z_1}
R(X,Y)_x,\ldots \in\mathfrak{so}(T_xM,g_x),
$$
where $X,Y,Z_1,Z_2,\ldots\in T_xM$.
\end{theorem}

A subspace $U\subset \mathbb{R}^{r,s}$ is called non-degenerate if
the restriction of the form~$g$ to this subspace is
non-degenerate. A Lie subgroup $G\subset \mathrm{O}(r,s)$ (or a
subalgebra $\mathfrak{g}\subset \mathfrak{so}(r,s)$) is called
called irreducible if it does not preserve any proper vector
subspace of  $\mathbb{R}^{r,s}$; $G$ (or $\mathfrak{g}$) is called
weakly irreducible if it does not preserve any proper
non-degenerate vector subspace of $\mathbb{R}^{r,s}$.

It is clear that a subalgebra $\mathfrak{g}\subset
\mathfrak{so}(r,s)$ is irreducible (resp. weakly irreducible) if
and only if the corresponding connected Lie subgroup  $G\subset
\operatorname{SO}(r,s)$ is irreducible (resp. weakly irreducible).
If a subgroup $G\subset \mathrm{O}(r,s)$ is irreducible, then it
is weakly irreducible. The converse holds only for positively and
negatively definite metrics~$g$.

Let us consider two pseudo-Riemannian manifolds $(M,g)$ and
$(N,h)$. Let ${x\in M}$,  $y\in N$, and let~$G_x$,~$H_y$ be the
corresponding  holonomy groups. The product of the manifolds $M\times
N$ is a  pseudo-Riemannian manifold with respect to the
metric~$g+h$.  A pseudo-Riemannian manifold is called (locally)
indecomposable if it is not a (local) product of pseudo-Riemannian
manifolds. Denote by~$F_{(x,y)}$ the holonomy group of the
manifold $M\times N$ at the point $(x,y)$. It holds
$F_{(x,y)}=G_x\times H_y$. This statement has the following
inverse one.

\begin{theorem}
\label{th6} Let $(M,g)$ be a pseudo-Riemannian manifold, and $x\in
M$. Suppose that the restricted  holonomy group~$G^0_x$ is not
weakly irreducible. Then the space~$T_xM$ admits an orthogonal
 decomposition (with respect to~$g_x$) into the direct sum of non-degenerate subspaces:
$$
T_xM=E_0\oplus E_1\oplus\cdots\oplus E_t,
$$
at that, $G^0_x$ acts trivially on~$E_0$, $G^0_x(E_i)\subset E_i$
($i=1,\dots,t$), and $G^0_x$ acts weakly irreducibly on~$E_i$
($i=1,\dots,t$). There exist a flat pseudo-Riemannian submanifold
$N_0\subset M$ and locally indecomposable pseudo-Riemannian
submanifolds $N_1,\dots,N_t\subset M$ containing the point~$x$
such that $T_xN_i=E_i$ ($i=0,\dots,t$). There exist open subsets
 $U\subset M$, $U_i\subset N_i$
($i=0,\dots,t$) containing the point~$x$ such that
$$
U=U_0\times U_1\times\cdots\times U_t,\qquad g\big|_{TU\times
TU}=g\big|_{TU_0\times TU_0}+g\big|_{TU_1\times TU_1}+
\cdots+g\big|_{TU_t\times TU_t}.
$$
Moreover, there exists a decomposition
$$
G^0_x=\{\operatorname{id}\}\times H_1\times\cdots\times H_t,
$$
where $H_i=G^0_x\big|_{E_i}$ are normal Lie subgroups in~$G^0_x$
($i=1,\dots,t$).

Furthermore, if the manifold~$M$ is simply connected and complete,
then there exists a global decomposition
$$
M=N_0\times N_1\times\cdots\times N_t.
$$
\end{theorem}

Local statement of this theorem for the case of Riemannian
manifolds proved Borel and Lichnerowicz~\cite{31}. The global
statement for the case of Riemannian manifolds proved
de~Rham~\cite{48}. The statement of the theorem for
pseudo-Riemannian manifolds proved Wu~\cite{122}.

In~\cite{70} algorithms for finding the de~Rham decomposition for
Riemannian manifolds and the Wu decomposition for  Lorentzian
manifold are given. For that the analysis of the parallel bilinear
forms on the manifold is used.

From Theorem~\ref{th6} it follows that a pseudo-Riemannian
manifold is locally indecomposable if and only if its restricted
holonomy group is weakly irreducible.

It is important to note that the Lie algebras of the Lie
groups~$H_i$ from Theorem~\ref{th6} are Berger algebras. The next
theorem is the  algebraic version of Theorem~\ref{th6}.

\begin{theorem}
\label{th7} Let $\mathfrak{g}\subset\mathfrak{so}(p,q)$ be a
Berger  subalgebra that is not irreducible. Then there exists the
following orthogonal decomposition
$$
\mathbb{R}^{p,q}=V_0\oplus V_1\oplus\cdots\oplus V_r
$$
and the decomposition
$$
\mathfrak{g}=\mathfrak{g}_1\oplus\cdots\oplus \mathfrak{g}_r
$$
into a direct sum of ideals such that~$\mathfrak{g}_i$
annihilates~$V_j$ for $i\ne j$ and  $\mathfrak{g}_i\subset
\mathfrak{so}(V_i)$ is a weakly irreducible Berger subalgebra.
\end{theorem}

\subsection{Connected irreducible holonomy groups of Riemannian and
pseudo-Riemannian manifolds} \label{ssec2.3} In the previous
subsection we have seen that the classification problem for the
subalgebras $\mathfrak{g}\subset\mathfrak{so}(r,s)$ with the
property $L(\mathscr{R}(\mathfrak{g}))=\mathfrak{g}$ can be
reduced to the classification problem for  the weakly irreducible
subalgebra $\mathfrak{g}\subset\mathfrak{so}(r,s)$ satisfying this
property. For the subalgebra $\mathfrak{g}\subset\mathfrak{so}(n)$
the weak irreducibility is equivalent to the irreducibility.
Recall that a pseudo-Riemannian manifold $(M,g)$ is called locally
symmetric if its curvature tensor satisfies the equality $\nabla
R=0$. For any locally symmetric Riemannian manifold there exists a
simply connected Riemannian manifold  with the same restricted
holonomy group.  Simply connected Riemannian symmetric spaces were
classified by \'E.~Cartan~\cite{25},~\cite{43},~\cite{82}. If the
holonomy group of such a space is irreducible, then it coincides
with the isotropy representation. Thus connected irreducible
holonomy groups of locally symmetric Riemannian manifolds are
known.

It is important to note that there exists a one-to-one
correspondence between simply connected indecomposable symmetric
Riemannian manifolds~$(M,g)$ and simple $\mathbb{Z}_2$-graded Lie
algebras $\mathfrak{g}=\mathfrak{h}\oplus\mathbb{R}^n$ such that
$\mathfrak{h}\subset\mathfrak{so}(n)$. The subalgebra
$\mathfrak{h}\subset\mathfrak{so}(n)$ coincides with the holonomy
algebra of the manifolds~$(M,g)$. The space $(M,g)$ can be
reconstructed using its  holonomy algebra
$\mathfrak{h}\subset\mathfrak{so}(n)$ and the value
$R\in\mathscr{R}(\mathfrak{h})$ of curvature tensor of the space
$(M,g)$ at some point. For that let us define the Lie algebra
structure on the vector space
$\mathfrak{g}=\mathfrak{h}\oplus\mathbb{R}^n$ in the following
way:
$$
[A,B]=[A,B]_{\mathfrak{h}},\quad [A,X]=AX,\quad
[X,Y]=R(X,Y),\qquad A,B\in\mathfrak{h},\quad X,Y\in\mathbb{R}^n.
$$
Then, $M=G/H$, where $G$ is a simply connected Lie group with the
Lie algebra~$\mathfrak{g}$, and $H\subset G$ the connect Lie
subgroup corresponding to the
subalgebra~$\mathfrak{h}\subset\mathfrak{g}$.

In 1955 Berger obtained a list of possible connected irreducible
holonomy groups of Riemannian manifolds~\cite{23}.

\begin{theorem}
\label{th8} If $G\subset \operatorname{SO}(n)$ is a connected Lie
subgroup such that its Lie algebra
$\mathfrak{g}\subset\mathfrak{so}(n)$ satisfies the condition
$L(\mathscr{R}(\mathfrak{g}))=\mathfrak{g}$, then either~$G$ is
the holonomy group of a locally symmetric Riemannian space, or~$G$
is one of the following groups: $\operatorname{SO}(n)$;
$\operatorname{U}(m)$, $\operatorname{SU}(m)$, $n=2m$;
$\operatorname{Sp}(m)$, $\operatorname{Sp}(m)\cdot
\operatorname{Sp}(1)$, $n=4m$; $\operatorname{Spin}(7)$, $n=8$;
$G_2$, $n=7$.
\end{theorem}

The initial  Berger list contained also the Lie group
$\operatorname{Spin}(9)\subset\operatorname{SO}(16)$. In~\cite{2}
D.~V.~Alekseevsky showed that Riemannian manifolds with the
holonomy group $\operatorname{Spin}(9)$ are locally symmetric. The
list of possible connected irreducible holonomy groups of not
locally symmetric Riemannian manifolds from Theorem~\ref{th8}
coincides with the list of connected Lie groups $G\subset
\operatorname{SO}(n)$ acting transitively on the unite sphere
$S^{n-1}\subset \mathbb{R}^n$ (if we exclude from the last list
the Lie groups $\operatorname{Spin}(9)$ and
$\operatorname{Sp}(m)\cdot T$, where $T$ is the circle). Having
observed that, in 1962 Simons obtained in~\cite{114} a direct
proof of the  Berger result. A more simple and geometric proof
very recently found
 Olmos~\cite{107}.

The proof of the Berger Theorem~\ref{th8} is based on the
classification of the irreducible real representations of the real
compact Lie algebras. Each such representation can be obtained
from the fundamental representations using the tensor products and
the decompositions into the irreducible components. The Berger
proof is reduced to the verification of the fact that such
representation (with several exceptions) cannot be the holonomy
representation: from the Bianchi identity it follows that
$\mathscr{R}(\mathfrak{g})=\{0\}$ if the representation contains
more then one tensor efficient. It remains to investigate only the
fundamental representations that are explicitly described by
\'E.~Cartan. Using complicated computations it is possible to show
that from the Bianchi identity it follows that either  $\nabla
R=0$, or $R=0$ except for the several exclusions given in
Theorem~\ref{th8}.

Examples of Riemannian manifolds with the holonomy groups
$\operatorname{U}(n/2)$, $\operatorname{SU}(n/2)$,
$\operatorname{Sp}(n/4)$ и $\operatorname{Sp}(n/4)\cdot
\operatorname{Sp}(1)$ constructed Calabi, Yau and Alekseevsky. In
1987 Bryant~\cite{40} constructed examples of Riemannian manifolds
with the holonomy groups $\operatorname{Spin}(7)$ and~$G_2$. This
completes the classification of the connected holonomy groups of
Riemannian manifolds.

Let us give the description of the geometric structures on
Riemannian manifolds with the holonomy groups form
Theorem~\ref{th8}.

{\leftskip = 16pt \parindent=-16pt

$\operatorname{SO}(n)$: This is the holonomy group of  Riemannian
manifolds of general position. There are no additional geometric
structures related to the holonomy group on such manifolds.

$\operatorname{U}(m)$ ($n=2m$): Manifolds with this holonomy group
are  K\"ahlerian, on each of these manifolds there exists a
parallel complex structure.

$\operatorname{SU}(m)$ ($n=2m$):  Each of the manifolds with this
holonomy group are K\"ahlerian and not Ricci-flat. They are called
special K\"ahlerian or Calabi-Yau manifolds.

$\operatorname{Sp}(m)$ ($n=4m$): On each  manifold with this
holonomy there exists a parallel quaternionic structure, i.e.
parallel complex structures~$I$,~$J$,~$K$ connected by the
relations $IJ=-JI=K$. These manifolds  are called
hyper-K\"ahlerian.

$\operatorname{Sp}(m)\cdot \operatorname{Sp}(1)$ ($n=4m$): On each
manifold with this holonomy group there exists a parallel
three-dimensional subbundle of the bundle of the endomorphisms of
the tangent spaces that locally is generated by a quaternionic
structure.

$\operatorname{Spin}(7)$ ($n=8$), $G_2$ ($n=7$): Manifolds with
these  holonomy groups are Ricci-flat. On a manifold with the
holonomy group $\operatorname{Spin}(7)$ there exists a parallel
4-form, on each manifold with the holonomy group~$G_2$ there
exists a parallel 3-form.

}

\noindent Thus indecomposable  Riemannian manifolds with special
(i.e., different from  $\operatorname{SO}(n)$) holonomy groups
have important geometric properties. Because of these properties
 Riemannian manifolds with special
holonomy groups found applications in theoretical physics (in
strings theory and M-theory)~\cite{45},~\cite{79},~\cite{89}.

The spaces~$\mathscr{R}(\mathfrak{g})$ for irreducible holonomy
algebras of Riemannian manifolds
$\mathfrak{g}\subset\mathfrak{so}(n)$ computed
Alekseevsky~\cite{2}. For $R\in \mathscr{R}(\mathfrak{g})$ define
the corresponding Ricci tensor asserting
$$
\operatorname{Ric}(R)(X,Y)=\operatorname{tr}(Z\mapsto R(Z,X)Y),
$$
$X,Y\in \mathbb{R}^n$. The space~$\mathscr{R}(\mathfrak{g})$
admits the following decomposition into the direct sum of
$\mathfrak{g}$-modules:
$$
\mathscr{R}(\mathfrak{g})=\mathscr{R}_0(\mathfrak{g})\oplus
\mathscr{R}_1(\mathfrak{g})\oplus\mathscr{R}^{\,\prime}(\mathfrak{g}),
$$
where~$\mathscr{R}_0(\mathfrak{g})$ consisting of the curvature
tensors with zero Ricci tensors, $\mathscr{R}_1(\mathfrak{g})$
consists of tensors annihilated by the Lie algebra~$\mathfrak{g}$
(this space is either trivial or one-dimension), and
$\mathscr{R}^{\,\prime}(\mathfrak{g})$ is the complement to these
two subspaces. If
$\mathscr{R}(\mathfrak{g})=\mathscr{R}_1(\mathfrak{g})$, then each
Riemannian manifold with the holonomy algebra
$\mathfrak{g}\subset\mathfrak{so}(n)$ is locally symmetric. Such
subalgebras $\mathfrak{g}\subset\mathfrak{so}(n)$ are called
 \textit{symmetric Berger algebras}. The holonomy
algebras of irreducible Riemannian symmetric spaces are exhausted
by the  algebras~$\mathfrak{so}(n)$, $\mathfrak{u}(n/2)$,
$\mathfrak{sp}(n/4)\oplus\mathfrak{sp}(1)$ and by symmetric Berger
algebras $\mathfrak{g}\subset\mathfrak{so}(n)$. For the holonomy
algebras~$\mathfrak{su}(m)$, $\mathfrak{sp}(m)$, $G_2$
and~$\mathfrak{spin}(7)$ it holds
$\mathscr{R}(\mathfrak{g})=\mathscr{R}_0(\mathfrak{g})$, and this
shows that the manifolds with such holonomy algebras are
Ricci-flat. Next, for
$\mathfrak{g}=\mathfrak{sp}(m)\oplus\mathfrak{sp}(1)$ it holds
$\mathscr{R}(\mathfrak{g})=\mathscr{R}_0(\mathfrak{g})\oplus
\mathscr{R}_1(\mathfrak{g}),$ consequently the corresponding
manifolds are Einstein manifolds.

The next theorem, proven by Berger in 1955, gives the
classification of possible connected irreducible  holonomy groups
of pseudo-Riemannian manifolds~\cite{23}.

\begin{theorem}
\label{th9} If $G\subset \operatorname{SO}(r,s)$ is a connected
irreducible Lie subgroup such that its Lie algebra
$\mathfrak{g}\subset\mathfrak{so}(r,s)$ satisfies the condition
$L(\mathscr{R}(\mathfrak{g}))=\mathfrak{g}$, then either~$G$ is
the holonomy group of a locally symmetric  pseudo-Riemannian
space, or~$G$ is one of the following groups:
$\operatorname{SO}(r,s)$; $\operatorname{U}(p,q)$,
$\operatorname{SU}(p,q)$, $r=2p$, $s=2q$;
$\operatorname{Sp}(p,q)$, $\operatorname{Sp}(p,q)\cdot
\operatorname{Sp}(1)$, $r=4p$, $s=4q$;
$\operatorname{SO}(r,\mathbb{C})$, $s=r$;
$\operatorname{Sp}(p,\mathbb{R})\cdot
\operatorname{SL}(2,\mathbb{R})$, $r=s=2p$;
$\operatorname{Sp}(p,\mathbb{C})\cdot
\operatorname{SL}(2,\mathbb{C})$, $r=s=4p$;
$\operatorname{Spin}(7)$, $r=0$, $s=8$;
$\operatorname{Spin}(4,3)$, $r=s=4$;
$\operatorname{Spin}(7)^{\mathbb{C}}$, $r=s=8$; $G_2$, $r=0$,
$s=7$; $G_{2(2)}^*$, $r=4$, $s=3$; $G_2^\mathbb{C}$, $r=s=7$.
\end{theorem}

The proof of Theorem~\ref{th9} uses the fact that a subalgebra
$\mathfrak{g}\subset\mathfrak{so}(r,s)$ satisfies the condition
$L(\mathscr{R}(\mathfrak{g}))=\mathfrak{g}$ if and only if its
complexification $\mathfrak{g}(\mathbb{C})\subset
\mathfrak{so}(r+s,\mathbb{C})$ satisfies the condition
$L(\mathscr{R}(\mathfrak{g}(\mathbb{C})))=\mathfrak{g}(\mathbb{C})$.
In other words, in Theorem~\ref{th9} are listed connected real Lie
groups such that their Lie algebras exhaust the real forms of the
complexifications of the Lie algebras for the Lie groups from
Theorem~\ref{th8}.

In 1957 Berger~\cite{24} obtained a list of connected irreducible
holonomy groups of pseudo-Riemannian symmetric spaces (we do not
give this list here since it is too large).

\section{Weakly irreducible subalgebras in~$\mathfrak{so}(1,n+1)$}
\label{sec3}

In this section we give a geometric interpretation from~\cite{57}
of the classification by B\'erard-Bergery and Ikemakhen~\cite{21} of
weakly irreducible subalgebras in~$\mathfrak{so}(1,n+1)$.

We start to study holonomy algebras of Lorentzian manifolds.
Consider a connected Lorentzian manifold $(M,g)$ of dimension
$n+2\geqslant 4$. We identify the tangent space at some point of
the manifold $(M,g)$  with the Minkowski space
$\mathbb{R}^{1,n+1}$. We will denote the Minkowski metric
on~$\mathbb{R}^{1,n+1}$ by the symbol~$g$. Then the holonomy
algebra~$\mathfrak{g}$ of the manifold $(M,g)$ at that point is
identified with a subalgebra of the Lorentzian Lie algebra
$\mathfrak{so}(1,n+1)$. By Theorem~\ref{th6}, $(M,g)$ is not
locally a product of pseudo-Riemannian manifolds if and only if
its  holonomy algebra $\mathfrak{g}\subset\mathfrak{so}(1,n+1)$ is
weakly irreducible. Therefore we will assume that
$\mathfrak{g}\subset\mathfrak{so}(1,n+1)$ is weakly irreducible.
If~$\mathfrak{g}$ is irreducible, then
$\mathfrak{g}=\mathfrak{so}(1,n+1)$. This follows from the Berger
results. In fact, $\mathfrak{so}(1,n+1)$ does not contain any
proper irreducible subalgebra; direct geometric proofs of this
statement can be found in~\cite{50} and~\cite{33}. Thus we may
assume that $\mathfrak{g}\subset\mathfrak{so}(1,n+1)$ is weakly
irreducible and  not irreducible; then~$\mathfrak{g}$ preserves a
degenerate subspace $U\subset\mathbb{R}^{1,n+1}$ and also the
isotropic line $\ell=U\cap U^\bot\subset\mathbb{R}^{1,n+1}$. We
fix an arbitrary isotropic vector  $p\in\ell$, then
$\ell=\mathbb{R} p$. Let us fix some other isotropic  vector~$q$
such that $g(p,q)=1$. The subspace $E\subset\mathbb{R}^{1,n+1}$
orthogonal to the vectors~$p$ and~$q$ is Euclidean; usually we
will denote this  space by~$\mathbb{R}^n$. Let $e_1,\dots,e_n$ be
an orthogonal basis in~$\mathbb{R}^n$. We get the Witt basis
$p,e_1,\dots,e_n,q$ of the space~$\mathbb{R}^{1,n+1}$.

Denote by $\mathfrak{so}(1,n+1)_{\mathbb{R} p}$ the maximal
subalgebra in~$\mathfrak{so}(1,n+1)$ preserving the isotropic
line~$\mathbb{R} p$. The Lie algebra
$\mathfrak{so}(1,n+1)_{\mathbb{R} p}$ can be identified with the
following matrix Lie algebra:
$$
\mathfrak{so}(1,n+1)_{\mathbb{R} p}=\left\{\begin{pmatrix} a &X^t
& 0
\\
0 & A & -X
\\
0 & 0 & -a
\end{pmatrix}\biggm| a\in \mathbb{R},\ X\in \mathbb{R}^n,\
A \in \mathfrak{so}(n)\right\}.
$$
We identify the above matrix with the triple $(a,A,X)$. We obtain
the subalgebras~$\mathbb{R}$, $\mathfrak{so}(n)$, $\mathbb{R}^n$
in~$\mathfrak{so}(1,n+1)_{\mathbb{R}p}$. It is clear
that~$\mathbb{R}$ commutes with~$\mathfrak{so}(n)$,
and~$\mathbb{R}^n$ is an ideal; we also have$$
[(a,A,0),(0,0,X)]=(0,0,aX+AX).
$$

We get the decomposition\footnote{Let $\mathfrak{h}$ be a Lie
algebra. We write
$\mathfrak{h}=\mathfrak{h}_1\oplus\mathfrak{h}_2$ if
$\mathfrak{h}$ is the direct sum of the ideals
$\mathfrak{h}_1,\mathfrak{h}_2\subset\mathfrak{h}$. We write
$\mathfrak{h}=\mathfrak{h}_1\ltimes\mathfrak{h}_2$
if~$\mathfrak{h}$ is the direct sum of a subalgebra
$\mathfrak{h}_1\subset\mathfrak{h}$ and an ideal
$\mathfrak{h}_2\subset\mathfrak{h}$. In the corresponding
situations for the Lie groups we use the symbols~$\times$
and~$\rightthreetimes$.}
$$
\mathfrak{so}(1,n+1)_{\mathbb{R} p}=
(\mathbb{R}\oplus\mathfrak{so}(n))\ltimes\mathbb{R}^n.
$$
Each weakly irreducible not irreducible subalgebra
$\mathfrak{g}\subset \mathfrak{so}(1,n+1)$ is conjugated to a
 weakly irreducible subalgebra
in~$\mathfrak{so}(1,n+1)_{\mathbb{R} p}$.

Let $\operatorname{SO}^0(1,n+1)_{\mathbb{R} p}$ be the connected
Lie subgroup of the Lie group $\operatorname{SO}(1,n+1)$
preserving the isotropic line~$\mathbb{R} p$. The
subalgebras~$\mathbb{R}$, $\mathfrak{so}(n)$,
$\mathbb{R}^n\subset\mathfrak{so}(1,n+1)_{\mathbb{R}p}$ correspond
to the following Lie subgroups:
\begin{gather*}
\left\{\begin{pmatrix}
a & 0 & 0\\ 0 & \operatorname{id} & 0\\ 0 & 0 & 1/a\\
\end{pmatrix}\biggm|a\in \mathbb{R},\ a>0\right\},\qquad
\left\{\begin{pmatrix}
1 & 0 & 0\\ 0 & f & 0\\ 0 & 0 & 1 \\
\end{pmatrix}\biggm| f\in \operatorname{SO}(n)\right\},
\\
\left\{\begin{pmatrix}
1 &X^t & -X^tX/2 \\ 0 & \operatorname{id} & -X\\ 0 & 0& 1 \\
\end{pmatrix}\biggm|X\in \mathbb{R}^n\right\}\subset
\operatorname{SO}^0(1,n+1)_{\mathbb{R}p}.
\end{gather*}
We obtain the decomposition
$$
\operatorname{SO}^0(1,n+1)_{\mathbb{R} p}=
(\mathbb{R}^+\times\operatorname{SO}(n))\rightthreetimes\mathbb{R}^n.
$$

Recall that each subalgebra $\mathfrak{h}\subset\mathfrak{so}(n)$
is compact and there exists the decomposition
$$
\mathfrak{h}=\mathfrak{h}'\oplus\mathfrak{z}(\mathfrak{h}),
$$
wher $\mathfrak{h}'=[\mathfrak{h},\mathfrak{h}]$ is the commutant
of~$\mathfrak{h}$, and $\mathfrak{z}(\mathfrak{h})$ is the center
of~$\mathfrak{h}$~\cite{118}.

The next result belongs to B\'erard-Bergery and Ikemakhen~\cite{21}.

\begin{theorem}
\label{th10} A subalgebra $\mathfrak{g}\subset
\mathfrak{so}(1,n+1)_{\mathbb{R} p}$ is weakly irreducible if and
only if~$\mathfrak{g}$ is a Lie algebra of one of the following
types.

{\rm\textbf{Type~1}}:
$$
\mathfrak{g}^{1,\mathfrak{h}}=(\mathbb{R}\oplus\mathfrak{h})\ltimes
\mathbb{R}^n=\left\{\begin{pmatrix}
a &X^t & 0\\ 0 & A &-X \\ 0 & 0 & -a \\
\end{pmatrix}\biggm| a\in \mathbb{R},\ X\in \mathbb{R}^n,\
A \in \mathfrak{h}\right\},
$$
where $\mathfrak{h}\subset\mathfrak{so}(n)$ is a subalgebra.

{\rm\textbf{Type~2}}:
$$
\mathfrak{g}^{2,\mathfrak{h}}=\mathfrak{h}\ltimes\mathbb{R}^n=
\left\{\begin{pmatrix}
0 &X^t & 0\\ 0 & A &-X \\ 0 & 0 & 0 \\
\end{pmatrix}\biggm|X\in \mathbb{R}^n,\ A \in \mathfrak{h}\right\},
$$
where $\mathfrak{h}\subset\mathfrak{so}(n)$ is a subalgebra.

{\rm\textbf{Type~3}}:
\begin{align*}
\mathfrak{g}^{3,\mathfrak{h},\varphi}&=\{(\varphi(A),A,0)\mid
A\in\mathfrak{h}\}\ltimes\mathbb{R}^n
\\
&=\left\{\begin{pmatrix}
\varphi(A) &X^t & 0\\ 0 & A &-X \\ 0 & 0 & -\varphi(A) \\
\end{pmatrix}\biggm| X\in \mathbb{R}^n,\ A \in \mathfrak{h}\right\},
\end{align*}
where $\mathfrak{h}\subset\mathfrak{so}(n)$ is a subalgebra
satisfying the condition $\mathfrak{z}(\mathfrak{h})\ne\{0\}$, and
 $\varphi\colon\mathfrak{h}\to\mathbb{R}$ is a non-zero linear map
 with the property $\varphi\big|_{\mathfrak{h}'}=0$.

{\rm\textbf{Type~4}}:
\begin{align*}
\mathfrak{g}^{4,\mathfrak{h},m,\psi}&=\{(0,A,X+\psi(A))\mid
A\in\mathfrak{h},\ X\in \mathbb{R}^m\}
\\
&=\left\{\begin{pmatrix}
0 &X^t&\psi(A)^t & 0\\ 0 & A&0 &-X \\ 0 & 0 & 0 &-\psi(A) \\
0&0&0&0\\ \end{pmatrix}\biggm| X\in \mathbb{R}^{m},\ A\in
\mathfrak{h}\right\},
\end{align*}
where exists an orthogonal decomposition
$\mathbb{R}^n=\mathbb{R}^m\oplus\mathbb{R}^{n-m}$ such that
$\mathfrak{h}\subset\mathfrak{so}(m)$,
$\dim\mathfrak{z}(\mathfrak{h})\geqslant n-m$, and
$\psi\colon\mathfrak{h}\to \mathbb{R}^{n-m}$ is a surjective
linear map with the property $\psi\big|_{\mathfrak{h}'}=0$.
\end{theorem}

The subalgebra $\mathfrak{h}\subset\mathfrak{so}(n)$ associated
above with a  weakly irreducible subalgebra $\mathfrak{g}\subset
\mathfrak{so}(1,n+1)_{\mathbb{R} p}$ is called \textit{the
orthogonal part} of the Lie algebra~$\mathfrak{g}$.

The proof of this theorem given in~\cite{21} is algebraic and it
does not give any interpretation of the obtained algebras. We give
a geometric proof of this result together with an illustrative
interpretation.

\begin{theorem}
\label{th11} There exists a Lie groups isomorphism
$$
\operatorname{SO}^0(1,n+1)_{\mathbb{R}p}\simeq\operatorname{Sim}^0(n),
$$
where $\operatorname{Sim}^0(n)$ is the connected Lie group of the
similarity transformations of the Euclidean space~$\mathbb{R}^n$.
Under this isomorphism  weakly irreducible Lie subgroups
from~$\operatorname{SO}^0(1,n+\nobreak1)_{\mathbb{R} p}$
correspond to transitive Lie subgroups
in~$\operatorname{Sim}^0(n)$.
\end{theorem}

\begin{proof}
We consider the boundary~$\partial L^{n+1}$ of the Lobachevskian
space
$$
\partial L^{n+1}=\{\mathbb{R}X\mid X\in\mathbb{R}^{1,n+1},\
g(X,X)=0,\ X\ne 0\}
$$
as the set of lines of the isotropic cone
$$
C=\{X\in \mathbb{R}^{1,n+1}\mid g(X,X)=0\}.
$$
Let us identify~$\partial L^{n+1}$ with the~$n$-dimensional unite
sphere~$S^n$ in the following way. Consider the basis
$e_0,e_1,\dots,e_n,e_{n+1}$ of the space~$\mathbb{R}^{1,n+1}$,
where
$$
e_0=\frac{\sqrt{2}}{2}(p-q),\quad e_{n+1}=\frac{\sqrt{2}}{2}(p+q).
$$
Consider the vector  subspace $E_1=E \oplus \mathbb{R}
e_{n+1}\subset\mathbb{R}^{1,n+1}$. Each isotropic line intersects
the affine  subspace $e_0+E_1$ at a unique point. The intersection
 $(e_0+E_1)\cap C$ constitutes the set
$$
\{X\in \mathbb{R}^{1,n+1} \mid x_0=1,\
x_1^2+\cdots+x_{n+1}^2=1\},
$$
which is the $n$-dimensional sphere~$S^n$. This gives us the
identification $\partial L^{n+1}\simeq S^n$.

The group  $\operatorname{SO}^0(1,n+1)_{\mathbb{R}p}$ acts
on~$\partial L^{n+1}$ (as the group of conformal transformations)
and it preserves the point $\mathbb{R} p\in\partial L^{n+1}$,
i.e., $\operatorname{SO}^0(1,n+1)_{\mathbb{R} p}$ acts on the
Euclidean  space $\mathbb{R}^n\simeq
\partial L^{n+1}\setminus\{\mathbb{R} p\}$ as the group of similarity transformations. Indeed, the
computations show that the elements
$$
\begin{pmatrix}
a & 0 & 0\\ 0 & \operatorname{id} & 0\\ 0 & 0 & 1/a \\
\end{pmatrix},
\begin{pmatrix}
1 & 0 & 0\\ 0 & f & 0\\ 0 & 0 & 1 \\
\end{pmatrix},
\begin{pmatrix}
1 & X^t & -X^tX/2 \\ 0 & \operatorname{id} & -X\\ 0 & 0& 1 \\
\end{pmatrix}\in \operatorname{SO}^0(1,n+1)_{\mathbb{R}p}
$$
act on~$\mathbb{R}^n$ as the homothety  $Y\mapsto aY$, the special
orthogonal transformation $f\in \operatorname{SO}(n)$ and the
translation $Y\mapsto Y+X$, respectively. Such transformations
generate the Lie group $\operatorname{Sim}^0(n)$. This gives the
isomorphism
$\operatorname{SO}^0(1,n+1)_{\mathbb{R}p}\simeq\operatorname{Sim}^0(n)$.
Next, it is easy to show that a subgroup $G\subset
\operatorname{SO}^0(1,n+1)_{\mathbb{R} p}$ does not preserve any
proper non-degenerate subspace in~$\mathbb{R}^{1,n+1}$ if and only
if the corresponding subgroup $G\subset \operatorname{Sim}^0(n)$
does not preserve any proper affine subspace in~$\mathbb{R}^n$.
The last condition is equivalent to the transitivity of the action
of~$G$ on~$\mathbb{R}^n$~\cite{3},~\cite{7}.
\end{proof}

It remains to classify connected transitive Lie subgroups
in~$\operatorname{Sim}^0(n)$. This is easy to do  using the
results from~\cite{3},~\cite{7} (see~\cite{57}).

\begin{theorem}
\label{th12} A connected subgroup
$G\subset\operatorname{Sim}^0(n)$ is transitive  if and only
if~$G$ is conjugated to a group of one of the following types.

{\rm\textbf{Type~1}}: $G=(\mathbb{R}^+\times H)\rightthreetimes
\mathbb{R}^n$, where $H\subset \operatorname{SO}(n)$ is a
connected Lie subgroup.

{\rm\textbf{Type~2}}: $G=H\rightthreetimes \mathbb{R}^n$.

{\rm\textbf{Type~3}}: $G=(\mathbb{R}^\Phi \times
H)\rightthreetimes \mathbb{R}^n$, where $\Phi\colon
\mathbb{R}^+\to \operatorname{SO}(n)$ is a homomorphism and
$$
\mathbb{R}^\Phi=\{a\cdot\Phi(a)\mid a\in \mathbb{R}^+\}\subset
\mathbb{R}^+\times \operatorname{SO}(n)
$$
is a group of screw homotheties of~$\mathbb{R}^n$.

{\rm\textbf{Type~4}}: $G=(H\times U^{\Psi})\rightthreetimes W$,
where exists an orthogonal decomposition  $\mathbb{R}^n=U\oplus
W$, $H\subset \operatorname{SO}(W)$,  $\Psi\colon U\to
\operatorname{SO}(W)$ is an injective homomorphism, and
$$
U^{\Psi}=\{\Psi(u)\cdot u\mid u\in U\}\subset
\operatorname{SO}(W)\times U
$$
is a group of screw  isometries of~$\mathbb{R}^n$.
\end{theorem}

It is easy to show that the  subalgebras
$\mathfrak{g}\subset\mathfrak{so}(1,n+1)_{\mathbb{R}p}$,
corresponding to the subgroups $G\subset\operatorname{Sim}^0(n)$
from the last theorem exhaust the Lie algebras from
Theorem~\ref{th10}. In what follows we will denote the Lie algebra
$\mathfrak{so}(1,n+1)_{\mathbb{R} p}$ by~$\mathfrak{sim}(n)$.

\section{Curvature tensors and classification of Berger algebras}
\label{sec4}

In this section we consider the structure of the spaces  of the
curvature tensors~$\mathscr{R}(\mathfrak{g})$ for subalgebras
$\mathfrak{g}\subset\mathfrak{sim}(n)$. Together with the result
by Leistner~\cite{100} about the classification of weak
 Berger algebras this will give a classification of the
 Berger subalgebras $\mathfrak{g}\subset\mathfrak{sim}(n)$. Next we find the curvature tensor
 of the Walker manifolds, i.e., manifolds
with the holonomy algebras $\mathfrak{g}\subset\mathfrak{sim}(n)$.
The results of this section are published in~\cite{55},~\cite{66}.

\subsection{Algebraic curvature tensors and classification of Berger algebras}
\label{ssec4.1} By the investigation of the
space~$\mathscr{R}(\mathfrak{g})$ for subalgebras
$\mathfrak{g}\subset\mathfrak{sim}(n)$ appears the space
\begin{align}
\nonumber \mathscr{P}(\mathfrak{h})=\{P\in \operatorname{Hom}
(\mathbb{R}^n,\mathfrak{h})\mid{}& g(P(X)Y,Z)+g(P(Y)Z,X)
\\
\label{eq4.1} &+g(P(Z)X,Y)=0,\ X,Y,Z\in \mathbb{R}^n\},
\end{align}
where $\mathfrak{h}\subset\mathfrak{so}(n)$ is a subalgebra. The
space~$\mathscr{P}(\mathfrak{h})$ is called the space of weak
curvature tensors for~$\mathfrak{h}$. Denote
by~$L(\mathscr{P}(\mathfrak{h}))$ the vector  subspace
in~$\mathfrak{h}$ spanned by the elements of the form~$P(X)$ for
all $P\in\mathscr{P}(\mathfrak{h})$ and $X\in \mathbb{R}^n$. It is
easy to show~\cite{55},~\cite{100} that if  $R\in
\mathscr{R}(\mathfrak{h})$, then for each $Z\in \mathbb{R}^n$ it
holds $P(\,\cdot\,)=R(\,\cdot\,,Z)\in\mathscr{P}(\mathfrak{h})$.
By this reason the algebra~$\mathfrak{h}$ is called a weak Berger
algebra if it holds $L(\mathscr{P}(\mathfrak{h}))=\mathfrak{h}$.
The structure of the $\mathfrak{h}$-module on the
space~$\mathscr{P}(\mathfrak{h})$ is introduced in the natural
way:
$$
P_\xi(X)=[\xi,P(X)]- P(\xi X),
$$
where $P\in\mathscr{P}(\mathfrak{h})$, $\xi\in\mathfrak{h}$, $X\in
\mathbb{R}^n$. This implies that the subspace
$L(\mathscr{P}(\mathfrak{h}))\subset\mathfrak{h}$ is an ideal
in~$\mathfrak{h}$.

It is convenient to identify the Lie algebra
$\mathfrak{so}(1,n+1)$  with the space of  bivectors
$\wedge^2\mathbb{R}^{1,n+1}$ in such a way that
$$
(X\wedge Y)Z=g(X,Z)Y-g(Y,Z)X,\qquad X,Y,Z\in\mathbb{R}^{1,n+1}.
$$
Then the element $(a,A,X)\in\mathfrak{sim}(n)$ corresponds to the
bivector $-ap\wedge q+A-p\wedge X$, where
$A\in\mathfrak{so}(n)\simeq\wedge^2\mathbb{R}^n$.

The next theorem from~\cite{55} provides the structure of the
space of the curvature tensors for the weakly irreducible
subalgebras $\mathfrak{g}\subset\mathfrak{sim}(n)$.

\begin{theorem}
\label{th13} Each curvature tensor
$R\in\mathscr{R}(\mathfrak{g}^{1,\mathfrak{h}})$ is uniquely
determined by the elements
$$
\lambda\in\mathbb{R},\quad \vec{v}\in \mathbb{R}^n,\quad
R_0\in\mathscr{R}(\mathfrak{h}),\quad
P\in\mathscr{P}(\mathfrak{h}),\quad T\in\odot^2\mathbb{R}^n
$$
in the following way:
\begin{gather}
R(p,q)=-\lambda p\wedge q-p\wedge \vec{v},\quad
R(X,Y)=R_0(X,Y)+p\wedge (P(X)Y-P(Y)X), \label{eq4.2}
\\
R(X,q)=-g(\vec v,X)p\wedge q+P(X)-p\wedge T(X),\quad R(p,X)=0,
\label{eq4.3}
\end{gather}
$X,Y\in\mathbb{R}^n$. In particular, there exists an isomorphism
of the  $\mathfrak{h}$-modules
$$
\mathscr{R}(\mathfrak{g}^{1,\mathfrak{h}})\simeq\mathbb{R}\oplus
\mathbb{R}^n\oplus\odot^2\mathbb{R}^n\oplus\mathscr{R}(\mathfrak{h})
\oplus\mathscr{P}(\mathfrak{h}).
$$
Next,
\begin{align*}
\mathscr{R}(\mathfrak{g}^{2,\mathfrak{h}})&=
\{R\in\mathscr{R}(\mathfrak{g}^{1,\mathfrak{h}})\mid\lambda=0, \
\vec v=0\},
\\
\mathscr{R}(\mathfrak{g}^{3,\mathfrak{h},\varphi})&=
\{R\in\mathscr{R}(\mathfrak{g}^{1,\mathfrak{h}})\mid\lambda=0,\
R_0\in\mathscr{R}(\ker\varphi),\ g(\vec v,\,\cdot\,)=
\varphi(P(\,\cdot\,))\},
\\
\mathscr{R}(\mathfrak{g}^{4,\mathfrak{h},m,\psi})&=
\{R\in\mathscr{R}(\mathfrak{g}^{2,\mathfrak{h}})\mid
R_0\in\mathscr{R}(\ker\psi),\
\operatorname{pr}_{\mathbb{R}^{n-m}}{\!}\circ{} T=\psi\circ
P\}.\end{align*}
\end{theorem}

\begin{corollary}[\cite{55}]
\label{cor1} A weakly irreducible subalgebra
$\mathfrak{g}\subset\mathfrak{sim}(n)$ is a Berger  algebra  if
and only if its orthogonal part
$\mathfrak{h}\subset\mathfrak{so}(n)$ is a weak Berger algebra.
\end{corollary}

\begin{corollary}[\cite{55}]
\label{cor2} A weakly irreducible subalgebra
$\mathfrak{g}\subset\mathfrak{sim}(n)$ such that its orthogonal
part $\mathfrak{h}\subset\mathfrak{so}(n)$ is the holonomy algebra
of a Riemannian manifold is a Berger algebra.
\end{corollary}

Corollary~\ref{cor1} reduces the classification problem of the
 Berger algebras for Lorentzian
manifolds to the classification problem of the weak Berger
algebras.

\begin{theorem}[\cite{55}]
\label{th14} {\rm (I)} For each weak  Berger algebra
$\mathfrak{h}\subset\mathfrak{so}(n)$ there exists an orthogonal
decomposition
\begin{equation} \label{eq4.4}
\mathbb{R}^{n}=\mathbb{R}^{n_1}\oplus\cdots\oplus\mathbb{R}^{n_s}
\oplus\mathbb{R}^{n_{s+1}}
\end{equation}
and the corresponding decomposition of  $\mathfrak{h}$ into the
direct sum of ideals
\begin{equation}
\label{eq4.5}
\mathfrak{h}=\mathfrak{h}_1\oplus\cdots\oplus\mathfrak{h}_s\oplus\{0\}
\end{equation}
such that $\mathfrak{h}_i(\mathbb{R}^{n_j})=0$ for $i\ne j$,
$\mathfrak{h}_i\subset\mathfrak{so}(n_i)$ and the representation
of~$\mathfrak{h}_i$ in~$\mathbb{R}^{n_i}$ is irreducible.

{\rm (II)} Suppose that  $\mathfrak{h}\subset\mathfrak{so}(n)$ is
a subalgebra with the decomposition from the part~{\rm (I)}. Then
holds the equality
$$
\mathscr{P}(\mathfrak{h})=\mathscr{P}(\mathfrak{h}_1)\oplus
\cdots\oplus\mathscr{P}(\mathfrak{h}_s).
$$
\end{theorem}

B\'erard-Bergery and Ikemakhen~\cite{21} proved that the  orthogonal
part $\mathfrak{h}\subset\mathfrak{so}(n)$ of a  holonomy algebra
$\mathfrak{g}\subset\mathfrak{sim}(n)$ admits the decomposition of
the part~(I) of Theorem~\ref{th14}.

\begin{corollary}[\cite{55}]
\label{cor3} Suppose that  $\mathfrak{h}\subset\mathfrak{so}(n)$
is a  subalgebra  admitting the decomposition as in part~{\rm (I)}
of Theorem~{\rm\ref{th14}}. Then~$\mathfrak{h}$ is a weak Berger
algebra  if and only if the algebra~$\mathfrak{h}_i$ is a weak
Berger algebra for all $i=1,\dots,s$.
\end{corollary}

Thus it is enough to consider irreducible weak Berger algebras
 $\mathfrak{h}\subset\mathfrak{so}(n)$. It turns out that these
algebras are irreducible holonomy algebras of Riemannian
manifolds. This far non-trivial statement proved
Leistner~\cite{100}.

\begin{theorem}[\cite{100}]
\label{th15} An irreducible subalgebra
$\mathfrak{h}\subset\mathfrak{so}(n)$ is a weak Berger algebra if
and only if it is the  holonomy algebra of a Riemannian manifold.
\end{theorem}

We will discuss the proof of this theorem below in Section
\ref{sec6}. From Corollary~\ref{cor1} and Theorem~\ref{th15} we
get the classification of weakly irreducible not irreducible
Berger algebras $\mathfrak{g}\subset\mathfrak{sim}(n)$.

\begin{theorem}
\label{th16} A subalgebra
$\mathfrak{g}\subset\mathfrak{so}(1,n+1)$ is weakly irreducible
not irreducible  Berger algebra  if and only if~$\mathfrak{g}$ is
conjugated to one of the  subalgebras
$\mathfrak{g}^{1,\mathfrak{h}}$, $\mathfrak{g}^{2,\mathfrak{h}}$,
$\mathfrak{g}^{3,\mathfrak{h},\varphi},\mathfrak{g}^{4,\mathfrak{h},m,\psi}
\subset\mathfrak{sim}(n)$, where
$\mathfrak{h}\subset\mathfrak{so}(n)$ is the holonomy algebra of a
Riemannian manifold.
\end{theorem}

Let us turn back to the statement of Theorem~\ref{th13}. Note that
the elements determining
$R\in\mathscr{R}(\mathfrak{g}^{1,\mathfrak{h}})$ from
Theorem~\ref{th13} depend on the choice of the vectors
$p,q\in\mathbb{R}^{1,n+1}$. Consider a real number $\mu\ne 0$, the
vector $p'=\mu p$ and an arbitrary isotropic vector~$q'$ such that
$g(p',q')=1$. There exists a unique vector $W\in E$ such that
$$
q'=\frac{1}{\mu}\biggl(-\frac{1}{2}g(W,W)p+W+q\biggr).
$$
The corresponding space~$E'$ has the form
$$
E'=\{-g(X,W)p+X\mid X\in E\}.
$$
We will consider the map
$$
E\ni X\mapsto X'=-g(X,W)p+X\in E'.
$$
It is easy to show that the tensor~$R$ is determined by the
elements $\widetilde\lambda$, $\widetilde v$, $\widetilde R_0$,
$\widetilde P$,~$\widetilde T$, where, i.e., we have
\begin{equation}
\label{eq4.6}
\begin{gathered}
\widetilde \lambda=\lambda,\quad \widetilde{v}=\frac{1}{\mu}(\vec
v-\lambda W)',\quad \widetilde
P(X')=\frac{1}{\mu}(P(X)+R_0(X,W))',
\\
\widetilde R_0(X',Y')Z'=(R_0(X,Y)Z)'.
\end{gathered}
\end{equation}

Let $R\in\mathscr{R}(\mathfrak{g}^{1,\mathfrak{h}})$. The
corresponding Ricci tensor has the form:
\begin{alignat}{2}
\label{eq4.7} \operatorname{Ric}(p,q)&=\lambda,&\quad
\operatorname{Ric}(X,Y)&=\operatorname{Ric}(R_0)(X,Y),
\\
\label{eq4.8} \operatorname{Ric}(X,q)&=g(X,\vec
v-\widetilde{\operatorname{Ric}}(P)),&\quad
\operatorname{Ric}(q,q)&=-\operatorname{tr}T,
\end{alignat}
where
$\widetilde{\operatorname{Ric}}(P)=\displaystyle\sum_{i=1}^nP(e_i)e_i$.
The scalar curvature satisfies
$$
s=2\lambda+s_0,
$$
where $s_0$ is the scalar curvature of the  tensor~$R_0$.

The Ricci operator has the following form:
\begin{gather}
\label{eq4.9} \operatorname{Ric}(p)=\lambda p,\qquad
\operatorname{Ric}(X)=g(X,\vec{v}-\widetilde{\operatorname{Ric}}(P))p+
\operatorname{Ric}(R_0)(X),
\\
\label{eq4.10} \operatorname{Ric}(q)=-(\operatorname{tr} T)p-
\widetilde{\operatorname{Ric}}(P)+\vec{v}+\lambda q.
\end{gather}

\subsection{Curvature tensor of Walker manifolds}
\label{ssec4.2} Each Lorentzian manifold $(M,g)$ with the holonomy
algebra $\mathfrak{g}\subset\mathfrak{sim}(n)$  (locally) admits a
parallel distribution of isotropic lines~$\ell$. These manifolds
are called the Walker  manifolds~\cite{37},~\cite{120}.

The vector bundle $\mathscr{E}=\ell^\bot/\ell$ is called the
screen bundle. The holonomy algebra of the induced connection
in~$\mathscr{E}$ coincides with the orthogonal part
$\mathfrak{h}\subset\mathfrak{so}(n)$ of the holonomy algebra of
the manifold $(M,g)$.

On a Walker  manifold $(M,g)$ there exist local coordinates
 $v,x^1,\dots,x^n,u$ such that the metric~$g$ is of the from
\begin{equation}
\label{eq4.11} g=2\,dv\,du+h+2A\,du+H\,(du)^2,
\end{equation}
where $h=h_{ij}(x^1,\dots,x^n,u)\,dx^i\,d x^j$ is a family of
Riemannian metrics depending on the parameter~$u$,
$A=A_i(x^1,\dots,x^n,u)\,dx^i$ is a family of 1-forms depending
on~$u$, and~$H$ is a local function on~$M$.

Note that the holonomy algebra of the metric~$h$ is contained in
the orthogonal part $\mathfrak{h}\subset\mathfrak{so}(n)$ of the
holonomy algebra of the metric~$g$, but this inclusion can be
strict.

The vector field~$\partial_v$ defines the parallel distribution of
isotropic lines and it is recurrent, i.e., it holds
$$
\nabla\partial_v=\frac{1}{2}\partial_vH\,du\otimes\partial_v.
$$
Therefore the vector field~$\partial_v$ is proportional to a
parallel  vector field if and only if $d(\partial_vH\,du)=0$,
which is equivalent to the equalities
$$
\partial_v\partial_iH=\partial_v^2H=0.
$$
In this case the coordinates can be chosen in such a way that
$\nabla\partial_v=0$ and $\partial_vH=0$. The  holonomy algebras
of type~2 and~4 annihilate the vector~$p$, consequently the
corresponding  manifolds admit  (local) parallel isotropic
 vector fields, and the local coordinates can be chosen in such a way that  $\partial_vH=0$.
 In contrast, the holonomy algebras of types~1 and~3 do not annihilate this vector, and consequently
 the corresponding manifolds admit only recurrent isotropic vector
 fields, in this case it holds~$d(\partial_vH\,du)\ne 0$.

An important class of Walker manifolds represent pp-waves, which
are defined locally by~\eqref{eq4.11} with $A=0$,
$h=\displaystyle\sum_{i=1}^n(dx^i)^2$, and $\partial_vH=0$.
Pp-waves are precisely Walker manifolds with commutative holonomy
algebras
$\mathfrak{g}\subset\mathbb{R}^n\subset\mathfrak{sim}(n)$.

Boubel~\cite{32} constructed the coordinates
\begin{equation}
\label{eq4.12} v,x_1=(x_1^1,\dots,x_1^{n_1}),\dots,x_{s+1}=
(x_{s+1}^1,\dots,x_{s+1}^{n_{s+1}}),u,
\end{equation}
corresponding to the decomposition~\eqref{eq4.4}. This means that
\begin{gather}
\label{eq4.13} h=h_1+\cdots+h_{s+1},\quad
h_\alpha=\sum_{i,j=1}^{n_\alpha}h_{\alpha ij}\,dx_{\alpha}^i\,
dx_{\alpha}^j,\quad h_{s+1}=\sum_{i=1}^{n_{s+1}}(dx_{s+1}^i)^2,
\\
\nonumber A=\sum_{\alpha=1}^{s+1} A_\alpha,\quad
A_\alpha=\sum_{k=1}^{n_\alpha} A^\alpha_k\,dx^k_\alpha,\quad
A_{s+1}=0,
\\
\label{eq4.14} \frac{\partial}{\partial x^k_\beta}h_{\alpha ij}=
\frac{\partial}{\partial x^k_\beta}A^\alpha_i=0,\quad\text{если }
\beta\ne\alpha.
\end{gather}

Consider the field of frames
\begin{equation}
\label{eq4.15} p=\partial_v,\quad
X_i=\partial_i-A_i\partial_v,\quad
q=\partial_u-\frac{1}{2}H\partial_v.
\end{equation}
Consider the distribution $E$ generated by the vector fields
$X_1,\cdots,X_n$. The fibers of this distribution can be
identified with the tangent spaces to the Riemannian manifolds
with the Riemannian metrics $h(u)$. Denote by $R_0$ the tensor
corresponding to  the family of the curvature tensors of the
metrics $h(u)$ under this identification. Similarly denote by
$\operatorname{Ric}(h)$ the corresponding Ricci endomorphism
acting on sections of $E$. Now the curvature tensor~$R$ of the
metric~$g$ is uniquely determined by a function~$\lambda$, a
section $\vec{v}\in \Gamma(E)$, a symmetric field of endomorphisms
 $T\in
\Gamma(\operatorname{End}(E))$, $T^*=T$, the curvature tensor
$R_0=R(h)$ and by a tensor $P\in \Gamma(E^*\otimes
\mathfrak{so}(E))$. These tensors can be expressed in terms of the
coefficients of the metric~\eqref{eq4.11}. Let
$P(X_k)X_j=P_{jk}^iX_i$ and $T(X_j)=\displaystyle\sum_iT_{ij}X_j$.
Then
$$
h_{il}P_{jk}^l=g(R(X_k,q)X_j,X_i),\quad T_{ij}=-g(R(X_i,q)q,X_j).
$$
The direct computations show that
\begin{align}
\label{eq4.16} \lambda&=\frac{1}{2}\partial^2_vH,\quad
\vec{v}=\frac{1}{2}(\partial_i\partial_vH-A_i\partial^2_vH)h^{ij}X_j,
\\
\label{eq4.17}
h_{il}P_{jk}^l&=-\frac{1}{2}\nabla_kF_{ij}+\frac{1}{2}\nabla_k\dot{h}_{ij}
-\dot\Gamma^l_{kj}h_{li},
\\
\nonumber
T_{ij}&=\frac{1}{2}\nabla_i\nabla_jH-\frac{1}{4}(F_{ik}+\dot
h_{ik})(F_{jl}+ \dot
h_{jl})h^{kl}-\frac{1}{4}(\partial_vH)(\nabla_iA_j+\nabla_jA_i)
\\
\nonumber &\qquad-\frac{1}{2}(A_i\partial_j\partial_vH+
A_j\partial_i\partial_vH)-\frac{1}{2}(\nabla_i\dot
A_j+\nabla_j\dot A_i)
\\
\label{eq4.18}
&\qquad+\frac{1}{2}A_iA_j\partial_v^2H+\frac{1}{2}\ddot{h}_{ij}+
\frac{1}{4}\dot h_{ij}\partial_vH,
\end{align}
where
$$
F=dA,\quad F_{ij}=\partial_iA_j-\partial_jA_i,
$$
is the differential of the 1-form~$A$, and the covariant
derivatives are taken with respect to the metric~$h$, the dot
denotes the partial derivative with respect to the variable~$u$.
In the case of $h$, $A$ and $H$ independent of~$u$, the curvature
tensor of the metric~\eqref{eq4.11} is found in~\cite{76}.
In~\cite{76} is also found the Ricci tensor of an arbitrary
metric~\eqref{eq4.11}.

It is important to note that the Walker coordinates are not
defined canonically, e.g.,  significant is the observation
from~\cite{76} showing that if
$$
H=\lambda v^2+vH_1+H_0,\qquad \lambda\in\mathbb{R},\quad
\partial_vH_1=\partial_vH_0=0,
$$
then the coordinates transformation
$$
v\mapsto v-f(x^1,\dots,x^n,u),\quad x^i \mapsto x^i,\quad u
\mapsto u
$$
changes the metric~\eqref{eq4.11} in the following way:
\begin{equation}
\label{eq4.19} A_i\mapsto A_i+\partial_{i}f,\quad H_1\mapsto
H_1+2\lambda f,\quad H_0\mapsto H_0+H_1 f+\lambda f^2+2\dot f.
\end{equation}

\section{The spaces of weak curvature tensors}
\label{sec5}

Although Leistner proved that the subalgebras
$\mathfrak{h}\subset\mathfrak{so}(n)$ spanned by the images of the
elements from the space~$\mathscr{P}(\mathfrak{h})$ are exhausted
by the holonomy algebras of Riemannian spaces, he did not found
the spaces~$\mathscr{P}(\mathfrak{h})$. Here we give the result of
computations of these spaces from~\cite{59}, this gives the
complete structure of the space of the curvature tensors for the
holonomy algebras $\mathfrak{g}\subset\mathfrak{sim}(n)$.

Let  $\mathfrak{h}\subset\mathfrak{so}(n)$ be an irreducible
subalgebra. Consider the $\mathfrak{h}$-equivariant map
$$
\widetilde{\operatorname{Ric}}\colon\mathscr{P}(\mathfrak{h})
\to\mathbb{R}^n,\qquad
\widetilde{\operatorname{Ric}}(P)=\sum_{i=1}^nP(e_i)e_i.
$$
The definition of this map does not depend on the choice of the
orthogonal basis $e_1,\dots,e_n$ of the space~$\mathbb{R}^n$.
Denote by~$\mathscr{P}_0(\mathfrak{h})$ the kernel of the
map~$\widetilde{\operatorname{Ric}}$. Let
$\mathscr{P}_1(\mathfrak{h})$ be the orthogonal complement of this
space in~$\mathscr{P}(\mathfrak{h})$. Thus,
$$
\mathscr{P}(\mathfrak{h})=\mathscr{P}_0(\mathfrak{h})\oplus
\mathscr{P}_1(\mathfrak{h}).
$$
Since the  subalgebra $\mathfrak{h}\subset\mathfrak{so}(n)$ is
irreducible and the map $\widetilde{\operatorname{Ric}}$ is
$\mathfrak{h}$-equivariant, the  space
$\mathscr{P}_1(\mathfrak{h})$ is either trivial, or it is
isomorphic to~$\mathbb{R}^n$. The
spaces~$\mathscr{P}(\mathfrak{h})$ for
$\mathfrak{h}\subset\mathfrak{u}(n/2)$ are found in~\cite{100}.
In~\cite{59} we compute the spaces~$\mathscr{P}(\mathfrak{h})$ for
the remaining Riemannian holonomy algebras. The main result is
Table~\ref{tab1}, where are given the
spaces~$\mathscr{P}(\mathfrak{h})$ for all irreducible holonomy
algebras $\mathfrak{h}\subset\mathfrak{so}(n)$ of Riemannian
manifolds (for a compact Lie algebra $\mathfrak{h}$ the expression
$V_\Lambda$ denotes the irreducible representation of
$\mathfrak{h}$  given by the irreducible representation of the Lie
algebra $\mathfrak{h}\otimes \mathbb{C}$ with the highest weight
$\Lambda$; $(\odot^2(\mathbb{C}^m)^*\otimes\mathbb{C}^m)_0$
denotes the subspace
in~$\odot^2(\mathbb{C}^m)^*\otimes\mathbb{C}^m$ consisting of
tensors such that the contraction of the upper index with any down
index gives zero).

\begin{table}

\centering

\small \caption{Spaces $\mathscr{P}(\mathfrak{h})$ for irreducible
holonomy algebras of Riemannian manifolds
$\mathfrak{h}\subset\mathfrak{so}(n)$}

\renewcommand{\arraystretch}{1.2}
\medskip
\label{tab1}
\begin{tabular}{|c|c|c|c|}
\hline $\mathfrak{h}\subset\mathfrak{so}(n)$&
$\mathscr{P}_1(\mathfrak{h})$&$\mathscr{P}_0(\mathfrak{h})$&
$\dim\mathscr{P}_0(\mathfrak{h})$
\\ \hline
$\mathfrak{so}(2)$&$\mathbb{R}^2$&0&0
\\ \hline
$\mathfrak{so}(3)$&$\mathbb{R}^3$&$V_{4\pi_1}$&5
\\ \hline
$\mathfrak{so}(4)$&$\mathbb{R}^4$&$V_{3\pi_1+\pi_1'}\oplus
V_{\pi_1+3\pi_1'}$&16
\\ \hline
$\mathfrak{so}(n),\,n\geqslant
5$&$\mathbb{R}^n$&$V_{\pi_1+\pi_2}$&
$\dfrac{(n-2)n(n+2)}{3}\vphantom{\Biggl\}}$
\\ \hline
$\mathfrak{u}(m),\,n=2m\geqslant 4$&
$\mathbb{R}^n$&$(\odot^2(\mathbb{C}^m)^*\otimes\mathbb{C}^m)_0$&
$m^2(m-1)$
\\ \hline
$\mathfrak{su}(m),\, n=2m\geqslant 4$&0
&$(\odot^2(\mathbb{C}^m)^*\otimes\mathbb{C}^m)_0$&$m^2(m-1)$
\\ \hline
$\mathfrak{sp}(m)\oplus\mathfrak{sp}(1),\, n=4m\geqslant 8$&
$\mathbb{R}^n$&$\odot^3(\mathbb{C}^{2m})^*$&$\dfrac{m(m+1)(m+2)}{3}\vphantom{\Biggl\}}$
\\ \hline
$\mathfrak{sp}(m),\, n=4m\geqslant 8$&0&
$\odot^3(\mathbb{C}^{2m})^*$&$\dfrac{m(m+1)(m+2)}{3}\vphantom{\Biggl\}}$
\\ \hline
$G_2\subset\mathfrak{so}(7)$&0&$V_{\pi_1+\pi_2}$&64
\\ \hline
$\mathfrak{spin}(7)\subset\mathfrak{so}(8)$&0&$V_{\pi_2+\pi_3}$&112
\\ \hline
\begin{tabular}{l}
$\mathfrak{h}\subset\mathfrak{so}(n),\,n\geqslant 4$,\\[-1mm]
is a symmetric\\[-1mm]
Berger algebra
\end{tabular}&$\mathbb{R}^n$&0&0
\\ \hline
\end{tabular}
\end{table}

Consider the natural  $\mathfrak{h}$-equivariant map
$$
\tau\colon\mathbb{R}^n\otimes\mathscr{R}(\mathfrak{h})\to
\mathscr{P}(\mathfrak{h}),\qquad \tau(u\otimes R)=R(\,\cdot\,,u).
$$
The next theorem will be used to get explicit form of some
 $P\in\mathscr{P}(\mathfrak{h})$. The proof of the theorem follows from the results
 of the papers~\cite{2},~\cite{100}
and Table~\ref{tab1}.

\begin{theorem}
\label{th17} For an arbitrary irreducible subalgebra
$\mathfrak{h}\subset\mathfrak{so}(n)$, $n\geqslant 4$, the
$\mathfrak{h}$-equivariant map $\tau\colon\mathbb{R}^n\otimes
\mathscr{R}(\mathfrak{h})\to\mathscr{P}(\mathfrak{h})$ is
surjective. Moreover,
$\tau(\mathbb{R}^n\otimes\mathscr{R}_0(\mathfrak{h}))=
\mathscr{P}_0(\mathfrak{h})$ and $\tau(\mathbb{R}^n\otimes
\mathscr{R}_1(\mathfrak{h}))=\mathscr{P}_1(\mathfrak{h})$.
\end{theorem}

Let $n\geqslant 4$, and $\mathfrak{h}\subset\mathfrak{so}(n)$ be
an irreducible subalgebra. From Theorem~\ref{th17} it follows that
an arbitrary $P\in\mathscr{P}_1(\mathfrak{h})$ can be written in
the form~$R(\,\cdot\,,x)$, where $R\in\mathscr{R}_0(\mathfrak{h})$
and~$x\in\mathbb{R}^n$. Similarly, any
$P\in\mathscr{P}_0(\mathfrak{h})$ can be represented in the form
 $\displaystyle\sum_i
R_i(\,\cdot\,,x_i)$ for some $R_i\in\mathscr{R}_1(\mathfrak{h})$
and  $x_i\in\mathbb{R}^n$.

\textbf{The explicit form of some
$P\in\mathscr{P}(\mathfrak{h})$}. Using the results obtained above
and results from~\cite{2}, we can now find explicitly the
spaces~$\mathscr{P}(\mathfrak{h})$.

From the results of the paper~\cite{100} it follows that
$$
\mathscr{P}(\mathfrak{u}(m))\simeq\odot^2(\mathbb{C}^m)^*\otimes\mathbb{C}^m.
$$
Let us give the explicit form of this isomorphism. Let
$$
S\in
\odot^2(\mathbb{C}^m)^*\otimes\mathbb{C}^m\subset(\mathbb{C}^m)^*
\otimes \mathfrak{gl}(m,\mathbb{C}).
$$
Consider the identification
$$ \mathbb{C}^m=\mathbb{R}^{2m}=\mathbb{R}^m\oplus
i\mathbb{R}^m
$$
and chose a basis $e_1,\dots,e_m$ of the space~$\mathbb{R}^m$.
Define the complex numbers~$S_{abc}$, $a,b,c=1,\dots,m$, such that
$$
S(e_a)e_b=\sum_c S_{acb}e_c.
$$
We have $S_{abc}=S_{cba}$. Define the map
$S_1\colon\mathbb{R}^{2m}\to\mathfrak{gl}(2m,\mathbb{R})$ by the
conditions
$$
S_1(e_a)e_b=\sum_c \overline {S_{abc}}e_c,\qquad
S_1(ie_a)=-iS_1(e_a),\qquad S_1(e_a)ie_b=iS_1(e_a)e_b.
$$
It is easy to check that
$$
P=S-S_1\colon\mathbb{R}^{2m}\to\mathfrak{gl}(2m,\mathbb{R})
$$
belongs to~$\mathscr{P}(\mathfrak{u}(n))$ and each element of the
space~$\mathscr{P}(\mathfrak{u}(n))$ is of this form. The obtained
element belongs to the space~$\mathscr{P}(\mathfrak{su}(n))$ if
and only if $\displaystyle\sum_b S_{abb}=0$ for all $a=1,\dots,m$,
i.e., $S\in (\odot^2(\mathbb{C}^m)^*\otimes\mathbb{C}^m)_0$. If
$m=2k$, i.e., $n=4k$, then~$P$ belongs to
$\mathscr{P}(\mathfrak{sp}(k))$ if and only if
$S(e_a)\in\mathfrak{sp}(2k,\mathbb{C})$, $a=1,\dots,m$, i.e.,
$$
S\in(\mathfrak{sp}(2k,\mathbb{C}))^{(1)}\simeq\odot^3(\mathbb{C}^{2k})^*.
$$

In~\cite{72} it is shown that each
$P\in\mathscr{P}(\mathfrak{u}(m))$ satisfies
$$
g(\widetilde{\operatorname{Ric}}(P),X)=
-\operatorname{tr}_\mathbb{C}P(JX),\qquad X\in\mathbb{R}^{2m}.
$$

In~\cite{2} it is shown that an arbitrary  $R\in
\mathscr{R}_1(\mathfrak{so}(n))
\oplus\mathscr{R}^{\,\prime}(\mathfrak{so}(n))$ has the form
${R=R_S}$, where $S\colon\mathbb{R}^{n}\to\mathbb{R}^{n}$ is a
symmetric linear map, and
\begin{equation} \label{eq5.1}
R_S(X,Y)=SX\wedge Y+ X\wedge SY.
\end{equation}
It is easy to check that
$$
\tau\bigl(\mathbb{R}^n,\mathscr{R}_1(\mathfrak{so}(n))\oplus
\mathscr{R}^{\,\prime}(\mathfrak{so}(n))\bigr)=\mathscr{P}(\mathfrak{so}(n)).
$$
This equality and~\eqref{eq5.1} show that the
space~$\mathscr{P}(\mathfrak{so}(n))$ is spanned by the
elements~$P$ of the form
$$
P(y)=Sy\wedge x+y\wedge Sx,
$$
where $x\in\mathbb{R}^{n}$ and $S\in\odot^2\mathbb{R}^n$ are
fixed, and $y\in\mathbb{R}^{n}$ is an arbitrary vector. For
such~$P$ we have
$\widetilde{\operatorname{Ric}}(P)=(\operatorname{tr} S-S)x$. This
means that the space~$\mathscr{P}_0(\mathfrak{so}(n))$ is spanned
by elements~$P$ of the form
$$
P(y)=Sy\wedge x,
$$
where $x\in\mathbb{R}^{n}$ and $S\in\odot^2\mathbb{R}^n$ satisfy
$\operatorname{tr}S=0$, $Sx=0$, and $y\in\mathbb{R}^{n}$ is an
arbitrary vector.

The isomorphism
$\mathscr{P}_1(\mathfrak{so}(n))\simeq\mathbb{R}^n$ is defined in
the following way: $x\in\mathbb{R}^n$ corresponds to the element
$P=x\wedge\cdot\in\mathscr{P}_1(\mathfrak{so}(n))$, i.e.,
$P(y)=x\wedge y$ for all $y\in\mathbb{R}^n$.

Each $P\in\mathscr{P}_1(\mathfrak{u}(m))$ has the form
$$
P(y)=-\frac{1}{2}g(Jx,y)J+\frac{1}{4}(x\wedge y+Jx\wedge Jy),
$$
where $J$ is the complex structure on~$\mathbb{R}^{2m}$, the
vector $x\in\mathbb{R}^{2m}$ is fixed, and the vector
$y\in\mathbb{R}^{2m}$ is arbitrary.

Each $P\in\mathscr{P}_1(\mathfrak{sp}(m)\oplus\mathfrak{sp}(1))$
has the form
$$
P(y)=-\frac{1}{2}\sum_{\alpha=1}^3g(J_\alpha x,y)J_\alpha+
\frac{1}{4}\biggl(x\wedge y+\sum_{\alpha=1}^3J_\alpha x\wedge
J_\alpha y\biggr),
$$
where $(J_1,J_2,J_3)$ is quaternionic structure
on~$\mathbb{R}^{4m}$, $x\in\mathbb{R}^{4m}$ is fixed, and
$y\in\mathbb{R}^{4m}$ is an arbitrary vector.

For the adjoint representation
$\mathfrak{h}\subset\mathfrak{so}(\mathfrak{h})$ of a simple
compact Lie algebra~$\mathfrak{h}$ different
from~$\mathfrak{so}(3)$, an arbitrary element
$P\in\mathscr{P}(\mathfrak{h})=\mathscr{P}_1(\mathfrak{h})$ has
the form
$$
P(y)=[x,y].
$$
If $\mathfrak{h}\subset\mathfrak{so}(n)$ is a symmetric Berger
 algebra, then
$$
\mathscr{P}(\mathfrak{h})=\mathscr{P}_1(\mathfrak{h})=
\{R(\,\cdot\,,x)\mid x\in\mathbb{R}^n\},
$$
where $R$ is a generator of the space
$\mathscr{R}(\mathfrak{h})\simeq\mathbb{R}$.

In general, let $\mathfrak{h}\subset\mathfrak{so}(n)$ be an
irreducible subalgebra, and $P\in\mathscr{P}_1(\mathfrak{h})$.
Then $\widetilde{\operatorname{Ric}}(P)\wedge\cdot\in
\mathscr{P}_1(\mathfrak{so}(n))$. Moreover, it is easy to check
that
$$
\widetilde{\operatorname{Ric}}\biggl(P+\frac{1}{n-1}
\widetilde{\operatorname{Ric}}(P)\wedge\cdot\biggr)=0,
$$
i.e.,
$$
P+\frac{1}{n-1}\widetilde{\operatorname{Ric}}(P)\wedge\cdot\in
\mathscr{P}_0(\mathfrak{so}(n)).
$$
Thus the inclusion
$$
\mathscr{P}_1(\mathfrak{h})\subset\mathscr{P}(\mathfrak{so}(n))=
\mathscr{P}_0(\mathfrak{so}(n))\oplus\mathscr{P}_1(\mathfrak{so}(n))
$$
has the form
$$
P\in P_1(\mathfrak{h})\mapsto\biggl(P+\frac{1}{n-1}
\widetilde{\operatorname{Ric}}(P)\wedge\cdot,
-\frac{1}{n-1}\widetilde{\operatorname{Ric}}(P)\wedge\cdot\biggr)
\in\mathscr{P}_0(\mathfrak{so}(n))\oplus\mathscr{P}_1(\mathfrak{so}(n)).
$$
This construction defines the tensor
$W=P+(1/(n-1))\widetilde{\operatorname{Ric}}(P)\wedge\cdot$
analogues to the Weyl tensor for $P\in\mathscr{P}(\mathfrak{h})$,
and this tensor is a component of the Weyl tensor of a Lorentzian
manifold.

\section{About the classification of weak Berger algebras}
\label{sec6}

One of the crucial instant of the classification of the holonomy
algebras of Lorentzian manifolds is the result by Leistner about
the classification of  irreducible weak Berger algebras
$\mathfrak{h}\subset\mathfrak{so}(n)$. Leistner classified all
such subalgebras and it turned out that the obtained list
coincides with the list of irreducible holonomy algebras of
Riemannian manifolds. The natural problem is to give a simple
direct proof to this fact.  In~\cite{68} we give such a proof for
the case of semisimple not simple Lie algebras
$\mathfrak{h}\subset\mathfrak{so}(n)$.

In~paper~\cite{55}, the first version of which  was published in
April  2003 on the web page \href{www.arXiv.org}{www.arXiv.org},
the Leistner theorem~\ref{th15} was proved for $n\leqslant 9$. For
that,  irreducible subalgebras
$\mathfrak{h}\subset\mathfrak{so}(n)$ with $n\leqslant 9$ were
listed (see Table~\ref{tab2}). The second column of the table
contains the irreducible holonomy algebras of Riemannian
manifolds.  The third column of the table contains algebras that
are not the holonomy algebras of Riemannian manifolds.

For a semisimple compact Lie algebra~$\mathfrak{h}$ we denote by
 $\pi^\mathbb{K}_{\Lambda_1,\dots,\Lambda_l}(\mathfrak{h})$
the image of the representation
$\pi^\mathbb{K}_{\Lambda_1,\dots,\Lambda_l}\colon \mathfrak{h}\to
\mathfrak{so}(n)$ that is determined by the complex representation
$\rho_{\Lambda_1,\dots,\Lambda_l}\colon
\mathfrak{h}(\mathbb{C})\to\mathfrak{gl}(U)$ given by the labels
$\Lambda_1,\dots,\Lambda_l$ on the Dynkin diagram (here
$\mathfrak{h}(\mathbb{C})$ is the complexification of the
algebra~$\mathfrak{h}$, $U$ is a complex vector  space),
$\mathbb{K}=\mathbb{R}$, $\mathbb{H}$ or~$\mathbb{C}$ if
$\rho_{\Lambda_1,\dots,\Lambda_l}$ is real, quaternionic or
complex, respectively. The symbol~$\mathfrak{t}$ denotes the
one-dimensional center.

\begin{table}[h]

\centering

\small

\renewcommand{\arraystretch}{1.3} \tabcolsep=0.55em

\caption{Irreducible subalgebras в~$\mathfrak{so}(n)$
($n\leqslant9$)} \label{tab2}

\smallskip

\begin{tabular}{|c|c|c|}
\hline $n$ & irreducible holonomy algebras & other irreducible
\\[-1mm]
& of $n$-dimensional Riemannian manifolds & subalgebras in
$\mathfrak{so}(n)$
\\ \hline
$n=1$ & &
\\ \hline
$n=2$ & $\mathfrak{so}(2)$ &
\\ \hline
$n=3$ & $\pi^\mathbb{R}_2(\mathfrak{so}(3))$ &
\\ \hline
$n=4$
&$\pi^\mathbb{R}_{1,1}(\mathfrak{so}(3)\oplus\mathfrak{so}(3))$,
$\pi^\mathbb{C}_1(\mathfrak{su}(2))$,
$\pi^\mathbb{C}_1(\mathfrak{su}(2))\oplus\mathfrak{t}$ &
\\ \hline
$n=5$ & $\pi^\mathbb{R}_{1,0}(\mathfrak{so}(5))$,
$\pi^\mathbb{R}_4(\mathfrak{so}(3))$ & \\ \hline $n=6$ &
$\pi^\mathbb{R}_{1,0,0}(\mathfrak{so}(6))$,
$\pi^\mathbb{C}_{1,0}(\mathfrak{su}(3))$,
$\pi^\mathbb{C}_{1,0}(\mathfrak{su}(3))\oplus\mathfrak{t}$ &
\\ \hline
$n=7$ & $\pi^\mathbb{R}_{1,0,0}(\mathfrak{so}(7))$,
$\pi^\mathbb{R}_{1,0}(\mathfrak{g}_2)$ &
$\pi^\mathbb{R}_6(\mathfrak{so}(3))$ \\ \hline $n=8$ &
$\pi^\mathbb{R}_{1,0,0,0}(\mathfrak{so}(8))$,
$\pi^\mathbb{C}_{1,0}(\mathfrak{su}(4))$,
$\pi^\mathbb{C}_{1,0}(\mathfrak{su}(4))\oplus\mathfrak{t}$, &
$\pi^\mathbb{C}_3(\mathfrak{so}(3))$,
\\
&$\pi^\mathbb{H}_{1,0}(\mathfrak{sp}(2))$,
$\pi^\mathbb{R}_{1,0,1}(\mathfrak{sp}(2)\oplus\mathfrak{sp}(1)),$
$\pi^\mathbb{R}_{0,0,1}(\mathfrak{so}(7)),$
&$\pi^\mathbb{C}_3(\mathfrak{so}(3))\oplus\mathfrak{t}$,
\\
& $\pi^\mathbb{R}_{1,3}(\mathfrak{so}(3)\oplus\mathfrak{so}(3))$,
$\pi^\mathbb{R}_{1,1}(\mathfrak{su}(3))$ &
$\pi^\mathbb{H}_{1,0}(\mathfrak{sp}(2))\oplus\mathfrak{t}$
\\ \hline
$n=9$ & $\pi^\mathbb{R}_{1,0,0,0}(\mathfrak{so}(9))$,
$\pi^\mathbb{R}_{2,2}(\mathfrak{so}(3)\oplus\mathfrak{so}(3))$&
$\pi^\mathbb{R}_8(\mathfrak{so}(3))$
\\
\hline
\end{tabular}
\end{table}

For algebras that are not the holonomy algebras of Riemannian
manifolds, with the help of a computer program the
spaces~$\mathscr{P}(\mathfrak{h})$ were found as the solutions of
the corresponding systems of linear equations. It turned out that
$$
\mathscr{P}(\pi^\mathbb{H}_{1,0}(\mathfrak{sp}(2)))=
\mathscr{P}(\pi^\mathbb{H}_{1,0}(\mathfrak{sp}(2))\oplus\mathfrak{t}),
$$
i.e., $L(\mathscr{P}(\pi^\mathbb{H}_{1,0}(\mathfrak{sp}(2))\oplus
\mathfrak{t}))=\pi^\mathbb{H}_{1,0}(\mathfrak{sp}(2))$, and
$\mathfrak{sp}(2)\oplus\mathfrak{t}$ is not a weak Berger algebra.
For other algebras of the third column we have
$\mathscr{P}(\mathfrak{h})=0$. Hence the Lie algebras from the
third column of Table~\ref{tab2} are not weak Berger algebras.

It turned out that by that time Leistner already proved
Theorem~\ref{th15} and published its proof as a preprint in the
cases when~$n$ is even and the representation
$\mathfrak{h}\subset\mathfrak{so}(n)$ is of complex type, i.e.,
$\mathfrak{h}\subset\mathfrak{u}(n/2)$. In this case
$\mathscr{P}(\mathfrak{h})\simeq
(\mathfrak{h}\otimes\mathbb{C})^{(1)}$, where
$(\mathfrak{h}\otimes\mathbb{C})^{(1)}$ is the first prolongation
of the  subalgebra
$\mathfrak{h}\otimes\mathbb{C}\subset\mathfrak{gl}(n/2,\mathbb{C})$.
Using this fact and the classification of irreducible
representations with non-trivial prolongations, Leistner showed
that each weak Berger subalgebra
$\mathfrak{h}\subset\mathfrak{u}(n/2)$  is the holonomy algebra of
a Riemannian manifold.

The case of subalgebras $\mathfrak{h}\subset\mathfrak{so}(n)$  of
real type (i.e. not of complex type) is much more difficult. In
this case Leistner considered the complexification
$\mathfrak{h}\otimes\mathbb{C}\subset\mathfrak{so}(n,\mathbb{C})$,
which is irreducible. Using the classification of irreducible
representations of complex semisimple Lie algebras, he found a
criteria in terms of weights for such representation
$\mathfrak{h}\otimes\mathbb{C}\subset\mathfrak{so}(n,\mathbb{C})$
to be a weak Berger algebra. Next Leistner considered case by case
simple Lie algebras $\mathfrak{h}\otimes\mathbb{C}$, and then
semisimple Lie algebras (the problem is reduced to the semisimple
Lie algebras of the form
$\mathfrak{sl}(2,\mathbb{C})\oplus\mathfrak{k}$, where
$\mathfrak{k}$ is simple, and again different possibilities for
$\mathfrak{k}$ were considered). The complete proof is published
in~\cite{100}.

We consider the case of semisimple not simple irreducible
subalgebras $\mathfrak{h}\subset\mathfrak{so}(n)$ with irreducible
complexification
$\mathfrak{h}\otimes\mathbb{C}\subset\mathfrak{so}(n,\mathbb{C})$.
In a simple way we show that it is enough to treat the case when
$\mathfrak{h}\otimes\mathbb{C}=\mathfrak{sl}(2,\mathbb{C})\oplus\mathfrak{k}$,
where $\mathfrak{k}\subsetneq\mathfrak{sp}(2m,\mathbb{C})$ is a
proper irreducible subalgebra, and the representation space is the
tensor product $\mathbb{C}^2\otimes\mathbb{C}^{2m}$. We show that
in this case $\mathcal{P}(\mathfrak{h})$ coincides with
$\mathbb{C}^2\otimes\mathfrak{g}_1$, where $\mathfrak{g}_1$ is the
first Tanaka prolongation of the non-positively graded Lie algebra
$$
\mathfrak{g}_{-2}\oplus\mathfrak{g}_{-1}\oplus\mathfrak{g}_0,
$$
here $\mathfrak{g}_{-2}=\mathbb{C}$,
$\mathfrak{g}_{-1}=\mathbb{C}^{2m}$,
$\mathfrak{g}_0=\mathfrak{k}\oplus\mathbb{C}
\operatorname{id}_{\mathbb{C}^{2m}}$, and the grading is defined
by the element $-\operatorname{id}_{\mathbb{C}^{2m}}$. We prove
that if~$\mathscr{P}(\mathfrak{h})$ is non-trivial,
then~$\mathfrak{g}_1$ is isomorphic~$\mathbb{C}^{2m}$, the second
Tanaka prolongation~$\mathfrak{g}_2$ is isomorphic
to~$\mathbb{C}$, and $\mathfrak{g}_3=0$. Then the full Tanaka
prolongation defines the simple $|2|$-graded complex Lie algebra
$$
\mathfrak{g}_{-2}\oplus\mathfrak{g}_{-1}\oplus\mathfrak{g}_0\oplus
\mathfrak{g}_{1}\oplus\mathfrak{g}_2.
$$
It is well known that simply connected indecomposable symmetric
Riemannian manifolds $(M,g)$ are in one-two-one correspondence
with simple $\mathbb{Z}_2$-graded Lie algebras
$\mathfrak{g}=\mathfrak{h}\oplus\mathbb{R}^n$ such that
$\mathfrak{h}\subset\mathfrak{so}(n)$. If the symmetric space is
quaternionic-K\"ahlerian, then
$\mathfrak{h}=\mathfrak{sp}(1)\oplus\mathfrak{f}\subset\mathfrak{so}(4k)$,
where $n=4k$, and $\mathfrak{f}\subset\mathfrak{sp}(k)$. The
complexification of the algebra
$\mathfrak{h}\oplus\mathbb{R}^{4k}$ coincides with
$(\mathfrak{sl}(2,\mathbb{C})
\oplus\mathfrak{k})\oplus(\mathbb{C}^2\otimes\mathbb{C}^{2k})$,
where $\mathfrak{k}=\mathfrak{f}\otimes\mathbb{C}\subset
\mathfrak{sp}(2k,\mathbb{C})$. Let~$e_1$, $e_2$ be the standard
basis of the  space~$\mathbb{C}^2$, and let
$$
F=\begin{pmatrix}0&0\\1&0\end{pmatrix},\qquad
H=\begin{pmatrix}1&0\\0&-1\end{pmatrix},\qquad
E=\begin{pmatrix}0&1\\0&0\end{pmatrix}
$$
be the basis of the Lie algebra $\mathfrak{sl}(2,\mathbb{C})$. We
get the following $\mathbb{Z}$-graded Lie algebra
$\mathfrak{g}\otimes\mathbb{C}$:
\begin{align*}
\mathfrak{g}\otimes\mathbb{C}&=\mathfrak{g}_{-2}\oplus\mathfrak{g}_{-1}
\oplus\mathfrak{g}_0\oplus\mathfrak{g}_1\oplus\mathfrak{g}_2
\\
&=\mathbb{C}F\oplus
e_2\otimes\mathbb{C}^{2k}\oplus(\mathfrak{k}\oplus \mathbb{C}
H)\oplus e_1\otimes\mathbb{C}^{2k}\oplus\mathbb{C} E.
\end{align*}
Conversaly, each such $\mathbb{Z}$-graded Lie algebra defines (up
to the duality) a simply connected quaternionic-K\"ahlerian
symmetric space. This gives the proof.

\section{Construction of metrics and the  classification theorem}
\label{sec7}

Above we have got the classification of  weakly irreducible Berger
algebras contained in~$\mathfrak{sim}(n)$. In this section we will
show that all these algebras can be realized as the holonomy
algebras of Lorentzian manifolds, we will noticeably simplify the
construction of the metrics from~\cite{56}. By that we complete
the classification of the  holonomy algebras of Lorentzian
manifolds.

The metrics realizing the Berger algebras of types~1 and~2
constructed B\'erard-Bergery and Ikemakhen~\cite{21}. These matrices
have the form
$$
g=2\,dv\,du+h+(\lambda v^2+H_0)\,(du)^2,
$$
where $h$ is a  Riemannian metric on~$\mathbb{R}^n$ with the
holonomy algebra $\mathfrak{h}\subset\mathfrak{so}(n)$,
$\lambda\in\mathbb{R}$, and $H_0$ is a generic function of the
variables $x^1,\dots, x^n$. If $\lambda\ne 0$, then the holonomy
algebra of this metric coincides
with~$\mathfrak{g}^{1,\mathfrak{h}}$; if $\lambda=0$, then the
holonomy algebra of the metric~$g$ coincides
with~$\mathfrak{g}^{2,\mathfrak{h}}$.

In~\cite{56} we gave a unified construction of metrics with all
possible holonomy algebras. Here we simplify this construction.

\begin{lemma}
\label{lem1} For an arbitrary holonomy algebra
$\mathfrak{h}\subset\mathfrak{so}(n)$ of a Riemannian manifold
there exists a $P\in\mathscr{P}(\mathfrak{h})$ such that the
vector space ${P(\mathbb{R}^n)\subset\mathfrak{h}}$ generates the
Lie algebra~$\mathfrak{h}$.
\end{lemma}

\begin{proof}
First we suppose that the  subalgebra
$\mathfrak{h}\subset\mathfrak{so}(n)$ is irreducible. If
$\mathfrak{h}$ is one of the holonomy algebras~$\mathfrak{so}(n)$,
$\mathfrak{u}(m)$, $\mathfrak{sp}(m)\oplus\mathfrak{sp}(1)$, then
for~$P$ it is enough to take one of the tensors described in
Section~\ref{sec5} for an arbitrary non-zero fixed
$X\in\mathbb{R}^n$. It is obvious that
$P(\mathbb{R}^n)\subset\mathfrak{h}$ generates the Lie
algebra~$\mathfrak{h}$. Similarly if
$\mathfrak{h}\subset\mathfrak{so}(n)$ is a symmetric Berger
algebra, then we can consider a non-zero $X\in\mathbb{R}^n$ and
put $P=R(X,\,\cdot\,)$, where $R$ is the curvature tensor of the
corresponding symmetric space. For~$\mathfrak{su}(m)$ we use the
isomorphism $\mathscr{P}(\mathfrak{su}(m))\simeq
(\odot^2(\mathbb{C}^m)^*\otimes\mathbb{C}^m)_0$ from
Section~\ref{sec5} and take $P$ determined by an element  $S\in
(\odot^2(\mathbb{C}^{m})^*\otimes\mathbb{C}^{m})_0$ that does not
belong to the space
$(\odot^2(\mathbb{C}^{m_0})^*\otimes\mathbb{C}^{m_0})_0$ for any
$m_0<m$. We do the same for~$\mathfrak{sp}(m)$.

The subalgebra $G_2\subset\mathfrak{so}(7)$ is generated by the
following matrices~\cite{15}:
\begin{alignat*}{4}
A_1&=E_{12}-E_{34},&\quad A_2&=E_{12}-E_{56},&\quad
A_3&=E_{13}+E_{24},&\quad A_4&=E_{13}-E_{67},
\\
A_5&=E_{14}-E_{23},&\quad A_6&=E_{14}-E_{57},&\quad
A_7&=E_{15}+E_{26},&\quad A_8&=E_{15}+E_{47},
\\
A_9&=E_{16}-E_{25},&\quad A_{10}&=E_{16}+E_{37},&\quad
A_{11}&=E_{17}-E_{36},&\quad A_{12}&=E_{17}-E_{45},
\\
A_{13}&=E_{27}-E_{35},&\quad A_{14}&=E_{27}+E_{46},&&&&
\end{alignat*}
where $E_{ij}\in\mathfrak{so}(7)$ ($i<j$) is the skew-symmetric
matrix such that
$(E_{ij})_{kl}=\delta_{ik}\delta_{jl}-\delta_{il}\delta_{jk}$.

Consider the linear map $P\in\operatorname{Hom}(\mathbb{R}^7,G_2)$
given by the formulas
\begin{alignat*}{4}
P(e_1)&=A_6,&\quad P(e_2)&=A_4+A_5,&\quad P(e_3)&=A_1+A_7,&\quad
P(e_4)&=A_1,
\\
P(e_5)&=A_4,&\quad P(e_6)&=-A_5+A_6,&\quad P(e_7)&=A_7.&&
\end{alignat*}
Using the computer it is easy to check that
$P\in\mathscr{P}(G_2)$, and the elements $A_1,A_4,A_5,A_6,A_7\in
G_2$ generate the Lie algebra~$G_2$.

The subalgebra $\mathfrak{spin}(7)\subset\mathfrak{so}(8)$ is
generated by the following matrices~\cite{15}:
\begin{alignat*}{4}
A_{1}&=E_{12}+E_{34},&\quad A_{2}&=E_{13}-E_{24},&\quad
A_{3}&=E_{14}+E_{23},&\quad A_{4}&=E_{56}+E_{78},
\\
A_{5}&=-E_{57}+E_{68},&\quad A_{6}&=E_{58}+E_{67},&\quad
A_{7}&=-E_{15}+E_{26},&\quad A_{8}&=E_{12}+E_{56},
\\
A_{9}&=E_{16}+E_{25},&\quad A_{10}&=E_{37}-E_{48},&\quad
A_{11}&=E_{38}+E_{47},&\quad A_{12}&=E_{17}+E_{28},
\\
A_{13}&=E_{18}-E_{27},&\quad A_{14}&=E_{35}+E_{46},&\quad
A_{15}&=E_{36}-E_{45},&\quad A_{16}&=E_{18}+E_{36},
\\
A_{17}&=E_{17}+E_{35},&\quad A_{18}&=E_{26}-E_{48},&\quad
A_{19}&=E_{25}+E_{38},&\quad A_{20}&=E_{23}+E_{67},
\\
A_{21}&=E_{24}+E_{57}.&&&&&&
\end{alignat*}
The linear map
$P\in\operatorname{Hom}(\mathbb{R}^8,\mathfrak{spin}(7))$, defined
by the formulas
\begin{alignat*}{4} P(e_1)&=0,&\quad
P(e_2)&=-A_{14},&\quad P(e_3)&=0,&\quad P(e_4)&=A_{21},
\\
P(e_5)&=A_{20},&\quad P(e_6)&=A_{21}-A_{18},&\quad
P(e_7)&=A_{15}-A_{16}, &\quad P(e_8)&=A_{14}-A_{17},
\end{alignat*}
belongs to the  space $\mathscr{P}(\mathfrak{spin}(7))$, and the
elements $A_{14},A_{15}-A_{16},A_{17},A_{18},$ $A_{20},A_{21}\in
\mathfrak{spin}(7)$ generate the Lie algebra $\mathfrak{spin}(7)$.

In the case of an arbitrary holonomy algebra
$\mathfrak{h}\subset\mathfrak{so}(n)$ the statement of the theorem
follows from Theorem~\ref{th14}.
\end{proof}

Consider an arbitrary holonomy algebra
$\mathfrak{h}\subset\mathfrak{so}(n)$ of a Riemannian manifold. We
will use the fact that $\mathfrak{h}$ ia a weak Berger algebra,
i.e., $L(\mathscr{P}(\mathfrak{h}))=\mathfrak{h}$. The initial
construction requires a fixation of enough number of elements
 $P_1,\dots,P_N\in\mathscr{P}(\mathfrak{h})$ such that their
 images generate~$\mathfrak{h}$. The just proven lemma allows to consider a single
 $P\in\mathscr{P}(\mathfrak{h})$. Recall that
for~$\mathfrak{h}$ the decompositions~\eqref{eq4.4}
and~\eqref{eq4.5} take a place. We will assume that the basis
$e_1,\dots,e_n$ of the space~$\mathbb{R}^n$ is concerned with the
decomposition~\eqref{eq4.4}. Let $m_0=n_1+\dots+n_s=n-n_{s+1}$.
Then, $\mathfrak{h}\subset\mathfrak{so}(m_0)$, and~$\mathfrak{h}$
does not annihilate any non-trivial subspace
in~$\mathbb{R}^{m_0}$. Note that in the case of the Lie algebras
$\mathfrak{g}^{4,\mathfrak{h},m,\psi}$ we have $0<m_0\leqslant m$.
Define the numbers~$P_{ji}^k$ such that $P(e_i)e_j=P_{ji}^ k e_
k$. Consider on~$\mathbb{R}^{n+2}$ the following metric:
\begin{equation}
\label{eq7.1} g=2\,dv\,du
+\sum^{n}_{i=1}(dx^i)^2+2A_i\,dx^i\,du+H\cdot(du)^2,
\end{equation}
where
\begin{equation}
\label{eq7.2} A_i=\frac{1}{3}(P^i_{jk}+P^i_{kj})x^j x^ k,
\end{equation}
and $H$ is a function that will depend on the type of the holonomy
algebra that we wish to construct.

For the Lie algebra $\mathfrak{g}^{3,\mathfrak{h},\varphi}$ define
the numbers $\varphi_{i}=\varphi(P(e_ i))$.

For the Lie algebra $\mathfrak{g}^{4,\mathfrak{h},m,\psi}$ define
the numbers~$\psi_{ij}$, $j=m+1,\dots,n$ such that
\begin{equation}
\label{eq7.3} \psi(P(e_i))=-\sum^{n}_{j=m+1}\psi_{ij}e_j.
\end{equation}

\begin{theorem}
\label{th18} The holonomy algebra~$\mathfrak{g}$ of the metric~$g$
at the point~$0$ depends on the function~$H$ in the following way:

\begin{center}
\begin{tabular}{|c|cl|}
\hline $H$& &$\mathfrak{g}$
\\ \hline
\rule{0pt}{18pt}$v^2+\displaystyle\sum_{i=m_0+1}^{n}(x^i)^2$&
&$\mathfrak{g}^{1,\mathfrak{h}}$
\\[4mm]\hline
$\displaystyle\sum_{i=m_0+1}^{n}(x^i)^2$&
&\rule{0pt}{18pt}$\mathfrak{g}^{2,\mathfrak{h}}$
\\[4mm]\hline
\rule{0pt}{18pt} $2v
\varphi_{i}x^i+\displaystyle\sum_{i=m_0+1}^{n}(x^i)^2$ &
&$\mathfrak{g}^{3,\mathfrak{h},\varphi}$
\\[4mm]\hline
\rule{0pt}{18pt}$2\displaystyle\sum_{j=m+1}^n\psi_{ij}x^i x^j+
\displaystyle\sum_{i=m_0+1}^{m}(x^i)^2$&
&$\mathfrak{g}^{4,\mathfrak{h},m,\psi}$
\\[4mm]
\hline
\end{tabular}
\end{center}
\end{theorem}

From Theorems~\ref{th16} and~\ref{th18} we get the main
classification Theorem.

\begin{theorem}
\label{th19} A subalgebra
$\mathfrak{g}\subset\mathfrak{so}(1,n+1)$ is weakly irreducible
not irreducible holonomy algebra of a Lorentzian manifold if and
only if~$\mathfrak{g}$ is conjugated to one of the following
subalgebras~$\mathfrak{g}^{1,\mathfrak{h}}$,
$\mathfrak{g}^{2,\mathfrak{h}}$,
$\mathfrak{g}^{3,\mathfrak{h},\varphi}$,
$\mathfrak{g}^{4,\mathfrak{h},m,\psi}\subset\mathfrak{sim}(n)$,
where $\mathfrak{h}\subset\mathfrak{so}(n)$ is the holonomy
algebra of a Riemannian manifold.
\end{theorem}

\begin{proof}[of Theorem~\ref{th18}]
Consider the field of frames~\eqref{eq4.15}. Let $X_p=p$ and
$X_q=q$. The indices~$a,b,c,\ldots$ will take all the values of
the indices of the basis vector fields. The components of the
connection~$\Gamma^c_{ba}$ are defined by the formula
$\nabla_{X_a}X_b=\Gamma^c_{ba}X_c$. The constructed metrics are
analytic. From the proof of Theorem~9.2 and~\cite{94} it follows
that~$\mathfrak{g}$ is generated by the elements of the form
$$
\nabla_{X_{a_\alpha}}\cdots\nabla_{X_{a_1}}R(X_a,X_b)(0)\in
\mathfrak{so}(T_0M,g_0)=\mathfrak{so}(1,n+1), \qquad
\alpha=0,1,2,\dots,
$$
where $\nabla$ is the Levi-Civita connection defined by the
metric~$g$, and~$R$ is the curvature tensor. The components of the
curvature tensor are defined by the equality
$$
R(X_a,X_b)X_c=\sum_dR^d_{cab}X_d.
$$
Note that the following recurrent formula takes a place:
\begin{align}
\nonumber \nabla_{a_\alpha}\cdots\nabla_{a_1}R^d_{cab}&=
X_{a_\alpha}\nabla_{a_{\alpha-1}}\cdots\nabla_{a_1}R^d_{cab}
\\
\label{eq7.4}
&\qquad+[\Gamma_{a_\alpha},\nabla_{X_{a_{\alpha-1}}}\cdots
\nabla_{X_{a_1}}R(X_a,X_b)]^d_c,
\end{align}
where~$\Gamma_{a_\alpha}$ denotes the operator with the
matrix~$(\Gamma_{ba_\alpha}^a)$. Since we consider the Walker
matric, it holds $\mathfrak{g}\subset\mathfrak{sim}(n)$.

Taking into account the said above it is not hard to find the
holonomy algebra~$\mathfrak{g}$. Let us make the computations for
the algebras of the fourth type. The proof for other types is
similar. Let $H=2\displaystyle\sum_{j=m+1}^n\psi_{ij}x^i
x^j+\displaystyle\sum_{i=m_0+1}^{m}(x^i)^2$. We must prove the
equality $\mathfrak{g}=\mathfrak{g}^{4,\mathfrak{h},m,\psi}$. It
is clear that $\nabla\partial_v=0$. Hence,
$\mathfrak{g}\subset\mathfrak{so}(n)\ltimes\mathbb{R}^n$.

The possibly non-zero Lie brackets of the basis vector fields are
the following:
\begin{gather*}
[X_i,X_j]=-F_{ij}p=2P_{ik}^jx^kp,\qquad [X_i,q]=C^p_{iq}p,
\\
C^p_{iq}=-\frac{1}{2}\partial_iH=\begin{cases}
-\displaystyle\sum_{j=m+1}^n\psi_{ij}x^j,&1\leqslant i\leqslant
m_0,
\\
-x^i, & m_0+1\leqslant i\leqslant m,
\\
-\psi_{ki}x^k,& m+1\leqslant i\leqslant n.
\end{cases}
\end{gather*}
Using this, it is easy to find the matrices of the
operators~$\Gamma_a$, namely, $\Gamma_p=0$,
\begin{alignat*}{2}
\Gamma_k&=\begin{pmatrix} 0& Y_k^t &0
\\
0&0&-Y_k\\0&0&0
\end{pmatrix},&\qquad
Y_k^t&=(P^k_{1i}x^i,\dots,P^k_{m_0i}x^i,0,\dots,0),
\\
\Gamma_q&=\begin{pmatrix} 0& Z^t &0
\\
0&(P^i_{jk}x^k)&-Z
\\
0&0&0
\end{pmatrix},&\qquad
Z^t&=-(C^p_{1q},\dots,C^p_{nq}).
\end{alignat*}

It is enough to compute the following components of the
 curvature tensor:
\begin{gather*}
R^k_{jiq}=P^ k_{ji},\quad R^k_{jil}=0,\quad R^k_{qij}=-P^i_{jk},
\\
R^j_{qjq}=-1,\quad m_0+1\leqslant j\leqslant m,\qquad
R^l_{qjq}=-\psi_{jl},\quad m+1\leqslant l\leqslant n.
\end{gather*}
This implies
\begin{gather*}
\operatorname{pr}_{\mathfrak{so}(n)}\bigl(R(X_i,q)(0)\bigr)=P(e_i),\quad
\operatorname{pr}_{\mathbb{R}^n}\bigl(R(X_i,q)(0)\bigr)=\psi(P(e_i)),
\\
\operatorname{pr}_{\mathbb{R}^n}\bigl(R(X_j,q)(0)\bigr)=-e_j,\qquad
m_0+1\leqslant j\leqslant m,
\\
\operatorname{pr}_{\mathbb{R}^n}\bigl(R(X_i,X_j)(0)\bigr)=P(e_j)e_i-P(e_i)e_j.
\end{gather*}
We get the inclusion
$\mathfrak{g}^{4,\mathfrak{h},m,\psi}\subset\mathfrak{g}$. The
formula~\eqref{eq7.4} and the induction allow to get the inverse
inclusion. The theorem is proved.
\end{proof}

Let us consider two \textit{examples}. Fom the proof of
Lemma~\ref{lem1} it follows that the holonomy algebra of the
metric
$$
g=2\,dv\,du+\sum^{7}_{i=1}(dx^i)^2+2\sum^{7}_{i=1}A_i\,dx^{i}\,du,
$$
where
\begin{align*}
A_1&=\frac{2}{3}(2x^2x^3+x^1x^4+2x^2x^4+2x^3x^5+x^5x^7),
\\
A_2&=\frac{2}{3}(-x^1x^3-x^2x^3-x^1x^4+2x^3x^6+x^6x^7),
\\
A_3&=\frac{2}{3}(-x^1x^2+(x^2)^2-x^3x^4-(x^4)^2-x^1x^5-x^2x^6),
\\
A_4&=\frac{2}{3}(-(x^1)^2-x^1x^2+(x^3)^2+x^3x^4),
\\
A_5&=\frac{2}{3}(-x^1x^3-2x^1x^7-x^6x^7),
\\
A_6&=\frac{2}{3}(-x^2x^3-2x^2x^7-x^5x^7),
\\
A_7&=\frac{2}{3}(x^1x^5+x^2x^6+2x^5x^6),
\end{align*} at the point $0\in\mathbb{R}^9$ coincides
with~$\mathfrak{g}^{2,G_2}\subset\mathfrak{so}(1,8)$. Similarly,
the  holonomy algebra of the metric
$$
g=2\,dv\,du+\sum^{8}_{i=1}(dx^i)^2+2\sum^{8}_{i=1}A_i\,dx^{i}\,du,
$$
where
\begin{alignat*}{2}
A_1&=-\frac{4}{3}x^7x^8,&\quad
A_2&=\frac{2}{3}((x^4)^2+x^3x^5+x^4x^6-(x^6)^2),
\\
A_3&=-\frac{4}{3}x^2x^5,&\quad
A_4&=\frac{2}{3}(-x^2x^4-2x^2x^6-x^5x^7+2x^6x^8),
\\
A_5&=\frac{2}{3}(x^2x^3+2x^4x^7+x^6x^7),&\quad
A_6&=\frac{2}{3}(x^2x^4+x^2x^6+x^5x^7-x^4x^8),
\\
A_7&=\frac{2}{3}(-x^4x^5-2x^5x^6+x^1x^8),&\quad
A_8&=\frac{2}{3}(-x^4x^6+x^1x^7),
\end{alignat*}
at the point $0\in\mathbb{R}^{10}$ coincides with
с~$\mathfrak{g}^{2,\mathfrak{spin}(7)}\subset\mathfrak{so}(1,9)$.

\section{Einstein equation}
\label{sec8}

In this section we consider the relation of the  holonomy algebras
and Einstein equation. We will find the holonomy algebras of
Einstein Lorentzian manifolds. Then we will show that in the case
of a non-zero cosmological constant, on a Walker manifold exist
special coordinates allowing to essentially simplify the Einstein
equation. Examples of Einstein metrics will be given. This topic
is motivated by the paper of theoretical physicists Gibbons and
Pope~\cite{76}. The results of this section are published
in~\cite{60},~\cite{61},~\cite{62},~\cite{74}.

\subsection[{Holonomy algebras of Einstein Lorentzian manifolds}] {Holonomy algebras of Einstein Lorentzian
manifolds}\label{ssec8.1}
Consider a  Lorentzian manifold $(M,g)$
with the holonomy algebra $\mathfrak{g}\subset\mathfrak{sim}(n)$.
First of all in~\cite{72} the following theorem was proved.

\begin{theorem}
\label{th20} Let $(M,g)$ be a locally indecomposable Lorentzian
Einstein manifold admitting a parallel distribution of isotropic
lines. Then the holonomy of $(M,g)$ is either of type 1 or 2. If
the cosmological constant of $(M,g)$ is non-zero, then the
holonomy algebra of $(M,g)$ is of type 1. If $(M,g)$ admits
locally a parallel isotropic vector field, then $(M,g)$ is
Ricci-flat. \end{theorem}

The classification complete the following two theorems
from~\cite{60}.

\begin{theorem}
\label{th21} Let $(M,g)$ be a locally indecomposable
$n+2$-dimensional Lorentzian manifold admitting a parallel
distribution of isotropic lines. If $(M,g)$ is Ricci-flat, then
one of the following statements holds.

{\rm (I)} The holonomy algebra $\mathfrak{g}$ of the manifold
$(M,g)$ is of type~1, and in the decomposition~\eqref{eq4.5} for
$\mathfrak{h}\subset\mathfrak{so}(n)$ at least one of the
subalgebras $\mathfrak{h}_i\subset\mathfrak{so}(n_i)$ coincides
with one of the Lie algebras: $\mathfrak{so}(n_i)$,
$\mathfrak{u}(n_i/2)$,
$\mathfrak{sp}(n_i/4)\oplus\mathfrak{sp}(1)$ or with a symmetric
 Berger algebra.

{\rm (II)} The holonomy algebra~$\mathfrak{g}$ of the manifold
$(M,g)$ is of type 2, and in the decomposition~\eqref{eq4.5} for
$\mathfrak{h}\subset\mathfrak{so}(n)$ each subalgebra
$\mathfrak{h}_i\subset\mathfrak{so}(n_i)$ coincides with one of
the Lie algebras: $\mathfrak{so}(n_i)$, $\mathfrak{su}(n_i/2)$,
$\mathfrak{sp}(n_i/4)$, $G_2\subset\mathfrak{so}(7)$,
$\mathfrak{spin}(7)\subset\mathfrak{so}(8)$.
\end{theorem}

\begin{theorem}
\label{th22} Let $(M,g)$ be a locally indecomposable
$n+2$-dimensional Lorentzian manifold admitting a parallel
distribution of isotropic lines. If $(M,g)$ is Einstein and not
Ricci-flat, then the holonomy algebra $\mathfrak{g}$ of $(M,g)$ is
of type 1, and in the decomposition~\eqref{eq4.5} for
$\mathfrak{h}\subset\mathfrak{so}(n)$ each subalgebra
$\mathfrak{h}_i\subset\mathfrak{so}(n_i)$ coincides with one of
the Lie algebras: $\mathfrak{so}(n_i)$, $\mathfrak{u}(n_i/2)$,
$\mathfrak{sp}(n_i/4)\oplus\mathfrak{sp}(1)$ or with a symmetric
 Berger algebra. Moreover, it holds $n_{s+1}=0$.
\end{theorem}

\subsection{Examples of Einstein metrics}
\label{ssec8.2} In this section we show the existence of  metrics
for each holonomy algebra obtained in the previous section.

From~\eqref{eq4.7} and~\eqref{eq4.8} it follows that the Einstein
equation
$$
\operatorname{Ric}=\Lambda g
$$
for the metric~\eqref{eq4.11} can be rewritten in notation of
Section~\ref{ssec4.2} in the following way:
\begin{equation}
\label{eq8.1} \lambda=\Lambda,\quad \operatorname{Ric}(h)=\Lambda
h,\quad \vec{v}=\widetilde{\operatorname{Ric}}(P),\quad
\operatorname{tr}T=0.
\end{equation}

First of all consider the metric~\eqref{eq4.11} such that $h$ is
an Einstein Riemannian metric with the holonomy
algebra~$\mathfrak{h}$ and non-zero cosmological
constant~$\Lambda$, and $A=0$. Let
$$
H=\Lambda v^2+H_0,
$$
where $H_0$ is a function depending on the coordinates
$x^1,\dots,x^n$. Then the first three equation from~\eqref{eq8.1}
hold true. From~\eqref{eq4.18} it follows that the last equation
has the form
$$
\Delta H_0=0,
$$
where
\begin{equation}
\label{eq8.2}
\Delta=h^{ij}(\partial_{i}\partial_{j}-\Gamma^k_{ij}\partial_k)
\end{equation}
 is the Laplace-Beltrami operator of the metric~$h$. Choosing a
generic harmonic function~$H_0$,  we get that the metric~$g$ is an
Einstein metric and it is indecomposable. From Theorem~\ref{th22}
it follows that
$\mathfrak{g}=(\mathbb{R}\oplus\mathfrak{h})\ltimes\mathbb{R}^n$.

Choosing in the same construction $\Lambda=0$, we get a Ricci-flat
metric with the holonomy algebra
$\mathfrak{g}=\mathfrak{h}\ltimes\mathbb{R}^n$.

Let us construct a Ricci-flat metric with the holonomy algebra
$\mathfrak{g}=(\mathbb{R}\oplus\mathfrak{h})\ltimes\mathbb{R}^n$,
where $\mathfrak{h}$ is as in Part~(I) of Theorem~\ref{th21}. For
that we use the construction of Section~\ref{sec7}. Consider
$P\in\mathscr{P}(\mathfrak{h})$
with~$\widetilde{\operatorname{Ric}}(P)\ne 0$. Recall that
$h_{ij}=\delta_{ij}$. Let
$$
H=vH_1+H_0,
$$
where~$H_1$ and $H_0$ are functions of the coordinates
$x^1,\dots,x^n$. The third equation from~\eqref{eq8.1} takes the
form
$$
\partial_k H_1=2\sum_iP^k_{ii},
$$
hence it is enough to take
$$
H_1=2\sum_{i,k}P^k_{ii}x^k.
$$
The last equation has the form
$$
\frac{1}{2}\sum_i\partial_i^2H_0-\frac{1}{4}\sum_{i,j}F^2_{ij}-
\frac{1}{2}H_1\sum_i\partial_iA_i-2A_i\sum_kP^i_{kk}=0.
$$
Note that
$$
F_{ij}=2P^j_{ik}x^k,\qquad
\sum_i\partial_iA_i=-2\sum_{i,k}P_{ii}^kx^k.
$$
We get an equation of the form
$\displaystyle\sum_i\partial^2_iH_0=K$, where $K$ is a polynomial
of degree two. A partial solution of this equation can be found in
the form
$$
H_0=\frac{1}{2}(x^1)^2K_2+\frac{1}{6}(x^1)^3 \partial_1 K_1+
\frac{1}{24}(x^i)^4(\partial_i)^2K,
$$
where
$$
K_1=K-\frac{1}{2}(x^i)^2(\partial_i)^2K,\qquad
K_2=K_1-x^1\partial_1K_1.
$$

In order to make the metric~$g$ indecomposable it is enough to add
to the obtained function~$H_0$ the harmonic function
$$
(x^1)^2+\cdots+(x^{n-1})^2-(n-1)(x^{n})^2.
$$
Since  $\partial_v\partial_iH\ne 0$, then the holonomy algebra of
the metric~$g$ is either of type~1 or~3. From Theorem~\ref{th21}
it follows that
$\mathfrak{g}=(\mathbb{R}\oplus\mathfrak{h})\ltimes\mathbb{R}^n$.

It is possible to construct in a similar way an example of a
Ricci-flat metric with the holonomy algebra
$\mathfrak{h}\ltimes\mathbb{R}^n$, where $\mathfrak{h}$ is as in
Part~(II) of Theorem~\ref{th21}. For that it is enough to consider
a $P\in\mathscr{P}(\mathfrak{h})$
with~$\widetilde{\operatorname{Ric}}(P)=0$, take $H_1=0$ and to
obtain a required~$H_0$.

We have proved the following theorem.

\begin{theorem}
\label{th23} Let $\mathfrak{g}$ be an algebra from
Theorem~{\rm\ref{th21}} or~{\rm\ref{th22}}, then there exists
an~$(n+2)$-dimensional Einstein Lorentzian manifold (or a Ricci
flat manifold) with the holonomy algebra~$\mathfrak{g}$.
\end{theorem}

\begin{example}
\label{ex1} In Section~\ref{sec7} we constructed metrics with the
holonomy algebras
$$
\mathfrak{g}^{2,G_2}\subset\mathfrak{so}(1,8) \quad
\text{and}\quad
\mathfrak{g}^{2,\mathfrak{spin}(7)}\subset\mathfrak{so}(1,9).
$$
Choosing in the just described way the function~$H$, we get
Ricci-flat metrics with the same holonomy algebras.
\end{example}

\subsection{Lorentzian manifolds with totally isotropic Ricci operator}
\label{ssec8.3} In the previous section we have seen that unlike
the case of  Riemannian manifold, Lorentzian manifolds with any of
the holonomy algebras are not automatically Ricci-flat nor
Einstein. Now we will see that never the less the Lorentzian
manifolds with some holonomy algebras  automatically satisfy a
weaker condition on the Ricci tensor.

A Lorentzian manifold $(M,g)$ is called \textit{totally
Ricci-isotropic} if the image of its Ricci operator is isotropic,
i.e.,
$$
g\bigl(\operatorname{Ric}(x),\operatorname{Ric}(y)\bigr)=0
$$
for all vector fields $X$ and $Y$. Obviously, any Ricci-flat
Lorentzian manifold is totally Ricci-isotropic.  If $(M,g)$ is a
spin manifold and it admits a parallel spinor, then it is totally
Ricci-isotropic~\cite{40},~\cite{52}.

\begin{theorem}
\label{th24} Let $(M,g)$ be a locally indecomposable
$n+2$-dimensional Lorentzian manifold admitting a parallel
distribution of isotropic lines. If $(M,g)$ is totally
Ricci-isotropic, then its holonomy algebra is the same  as in
Theorem~{\rm\ref{th21}}.
\end{theorem}

\begin{theorem}
\label{th25} Let $(M,g)$ be a locally indecomposable
$n+2$-dimensional Lorentzian manifold admitting a parallel
distribution of isotropic lines. If the holonomy algebra of
$(M,g)$ is of type 2 and in the decomposition~\eqref{eq4.5} of the
algebra $\mathfrak{h}\subset\mathfrak{so}(n)$ each subalgebra
$\mathfrak{h}_i\subset\mathfrak{so}(n_i)$ coincides with one of
the Lie algebras  $\mathfrak{su}(n_i/2)$, $\mathfrak{sp}(n_i/4)$,
$G_2\subset\mathfrak{so}(7)$,
$\mathfrak{spin}(7)\subset\mathfrak{so}(8)$, then the  manifold
$(M,g)$ is totally Ricci-isotropic.
\end{theorem}

Note that this theorem can  be also proved by the following
argument. Locally $(M,g)$ admits a spin structure.
From~\cite{72},~\cite{100} it follows that  $(M,g)$  admits
locally  parallel spinor fields, hence the manifold $(M,g)$ is
totally Ricci-isotropic.

\subsection{Simplification of the Einstein equation}
\label{ssec8.4} The Einstein equation for the
metric~\eqref{eq4.11} considered Gibbons and Pope~\cite{76}. First
of all the Einstein equation implies that
$$
H=\Lambda v^2+vH_1+H_0,\qquad
\partial_vH_1=\partial_vH_0=0.
$$
Next, it is equivalent to the system of equations
\begin{align}
&\Delta
H_0-\frac{1}{2}F^{ij}F_{ij}-2A^i\partial_{i}H_1-H_1\nabla^iA_i+
2\Lambda A^iA_i-2\nabla^i\dot{A}_i \nonumber
\\
\label{eq8.3} &\qquad+\frac{1}{2}\dot h^{ij}\dot
h_{ij}+h^{ij}\ddot h_{ij}+ \frac{1}{2}h^{ij}\dot h_{ij}H_1=0,
\\
\label{eq8.4} &\nabla^jF_{ij}+\partial_{i}H_1-2\Lambda
A_i+\nabla^j\dot h_{ij}-
\partial_{i}(h^{jk}\dot h_{jk})=0,
\\
\label{eq8.5} &\Delta H_1-2\Lambda \nabla^i A_i-\Lambda h^{ij}\dot
h_{ij}=0,
\\
\label{eq8.6} &\operatorname{Ric}_{ij}=\Lambda h_{ij},
\end{align}
where the operator~$\Delta$ is given by the formula~\eqref{eq8.2}.
The equations can be obtained considering the
equations~\eqref{eq8.1} and applying the formulas from
Section~\ref{ssec4.2}.

The Walker coordinates are not defined uniquely. E.g.,
Schimming~\cite{110} showed that if  $\partial_v H=0$, then the
coordinates can be chosen in such a way that $A=0$ and $H=0$. The
main theorem of this subsection gives a possibility to find
similar coordinates and by that to simplify  the Einstein equation
for the case $\Lambda\ne 0$.

\begin{theorem}
\label{th26} Let $(M,g)$ be a Lorentzian  manifold of dimension
$n+2$
 admitting a parallel distribution of isotropic lines. If $(M,g)$ is Einstein with a non-zero
 cosmological constant $\Lambda$, then  there exist local coordinates
 $v,x^1,\dots,x^n,u$ such that the metric~$g$ has the form
$$
g=2\,dv\,du+h+(\Lambda v^2+H_0)\,(du)^2,
$$
where $\partial_v H_0=0$, and  $h$ is a $u$-family of Einsteinа
Riemannian metrics with the cosmological constant~$\Lambda$
satisfying the equations
\begin{align}
\label{eq8.7} \Delta H_0+\frac{1}{2}h^{ij}\ddot{h}_{ij} &=0,
\\*
\label{eq8.8} \nabla^j\dot h_{ij}&=0,
\\*
\label{eq8.9} h^{ij}\dot{h}_{ij} &=0,
\\*
\label{eq8.10} \operatorname{Ric}_{ij}&=\Lambda h_{ij}.
\end{align}
Conversely, any such metric is Einstein.
\end{theorem}

Thus, we reduced the Einstein equation with $\Lambda\neq 0$ for
Lorentzian metrics to the problems of finding the families of
Einstein Riemannian metrics satisfying
Equations~\eqref{eq8.8},~\eqref{eq8.9} and functions $H_0$
satisfying  Equation~\eqref{eq8.7}.

\subsection{The case of dimension~4}
\label{ssec8.5} Let us consider the case of dimension~4, i.e.,
$n=2$. We will write $x=x^1$, $y=x^2$.

Ricci-flat Walker metrics in dimension 4 are found in \cite{92}.
They are given by $h=(dx)^2+(dy)^2$, $A_2=0$,
$H=-(\partial_xA_1)v+H_0$, where $A_1$ is a harmonic function and
$H_0$ is a solution of a Poisson equation.

 In \cite{102} all
4-dimensional Einstein Walker metrics with $\Lambda\ne 0$ are
described. The coordinates can be chosen in such a way that $h$ is
an independent of $u$ metric  of constant curvature. Next,
$A=W\,dz+\overline W\,d\overline z$, $W=i\partial_z L$, where
$z=x+iy$, and  $L$ is the $\mathbb{R}$-valued function given by
the formula
\begin{equation}
\label{eq8.11} L=2\operatorname{Re}\biggl(\phi\partial_z(\ln P_0)-
\frac{1}{2}\partial_z\phi\biggr), \qquad
2P_0^2=\biggl(1+\frac{\Lambda}{|\Lambda|}z\overline{z}\biggr)^2,
\end{equation}
where $\phi=\phi(z,u)$ is an arbitrary function holomorphic in $z$
and smooth in $u$. Finally, $H=\Lambda^2v+H_0$, and the function
$H_0=H_0(z,\overline{z},u)$ can be found in a similar way.

In this section we give examples of Einstein Walker metrics with
$\Lambda\ne 0$ such that $A=0$, and $h$ depends on $u$. The
solutions from \cite{102} are not useful for constructing examples
of such form, since ``simple'' functions $\phi(z,u)$ define
``complicated'' forms~$A$.  Similar examples can be constructed in
dimension 5, this case is discussed in~\cite{76},~\cite{78}.

Note that in dimension 2 (and~3) any Einstein Riemannian metric
  has constant sectional  curvature, hence any such metrics with the
  same $\Lambda$ are locally isometric, and the coordinates can be
  chosen in such a way that $\partial_u h=0$. As in \cite{102}, it is not hard to show
  that if $\Lambda>0$, then we may   assume that
 $h=\bigl((d x)^2+\sin^2 x\,(dy)^2\bigr)/\Lambda$,
$H=\Lambda v^2+H_0$, and the   Einstein equation is reduced to the
system
\begin{equation}
\label{eq8.12}
\begin{gathered}
\Delta_{S^2}f=-2f,\qquad \Delta_{S^2}H_0=2\Lambda\biggl(2
f^2-(\partial_xf)^2+ \frac{(\partial_yf)^2}{\sin^2x}\biggr),
\\
\Delta_{S^2}=\partial_x^2+\frac{\partial_y^2}{\sin^2 x}+\cot
x\,\partial_x.
\end{gathered}
\end{equation}
The function~$f$ determines the 1-form~$A$:
$$
A=-\frac{\partial_yf}{\sin x}\,dx+\sin x\,\partial_xf\,dy.
$$
Similarly, if $\Lambda<0$, then we consider
$$
h=\frac{1}{-\Lambda\cdot x^2}\bigl((d x)^2+(d y)^2\bigr)
$$
and get
\begin{equation}
\label{eq8.13}
\begin{gathered}
\Delta_{L^2}f=2f,\qquad \Delta_{L^2}H_0=-4\Lambda f^2-2\Lambda
x^2\bigl((\partial_xf)^2+ (\partial_yf)^2\bigr),
\\
\Delta_{L^2}=x^2(\partial_x^2+\partial_y^2),
\end{gathered}
\end{equation}
and $A=-\partial_yf\,dx+\partial_xf\,dy$. Thus in order to find
partial solutions of the system of
equtions~\eqref{eq8.7}--\eqref{eq8.10}, it is
 convenient first to find~$f$ and then, changing
 the coordinates, to get rid of the 1-form~$A$. After such a coordinate
 change, the metric
 $h$~does not depend on~$u$ if and only if $A$ is a Killing form
 for~$h$~\cite{76}. If
$\Lambda>0$, then this happens if and only if
$$
f=c_1(u)\sin x\sin y+c_2(u)\sin x\cos y+c_3(u)\cos x;
$$
for $\Lambda<0$ this is equivalent to the equality
$$
f=c_1(u)\frac{1}{x}+c_2(u)\frac{y}{x}+c_3(u)\frac{x^2+y^2}{x}\,.
$$

The functions $\phi(z,u)=c(u)$, $c(u)z$,~$c(u)z^2$
from~\eqref{eq8.11} determine the Killing form~$A$~\cite{74}. For
other functions~$\phi$, the form~$A$ has a complicated structure.
Let $g$ be an Einstein metric of the form~\eqref{eq4.11}
with~$\Lambda\ne 0$, $A=0$, and $H=\Lambda v^2+H_0$. The curvature
tensor~$R$ of the metric~$g$ has the form
$$
R(p,q)=\Lambda p\wedge q,\quad R(X,Y)=\Lambda X\wedge Y,\quad
R(X,q)=-p\wedge T(X),\quad R(p,X)=0.
$$
The metric~$g$ is indecomposable if and only if $T\ne 0$. In this
case the holonomy algebra coincides with~$\mathfrak{sim}(2)$.

For the Weyl tensor we have
\begin{alignat*}{2}
W(p,q)&=\frac{\Lambda}{3} p\wedge q,&\qquad
W(p,X)&=-\frac{2\Lambda}{3} p\wedge X,
\\
W(X,Y)&=\frac{\Lambda}{3} X\wedge Y,&\qquad
W(X,q)&=-\frac{2\Lambda}{3} X\wedge q -p\wedge T(X).
\end{alignat*}
In~\cite{81} it is shown that the Petrov type of the metric $g$ is
either~II or D~(and it may change
 from  point to point). From the Bel criteria it follows that~$g$
is of type~II at a point $m\in M$ if and only if $T_m\ne 0$,
otherwise~$g$ is of type~D. Since the endomorphism~$T_m$ is
symmetric and trace-free, it is either zero or it has rank~2.
Consequently, $T_m=0$ if and only if $\det T_m=\nobreak0$.

\begin{example}
\label{ex2} Consider the function $f=c(u)x^2$, then
$A=2xc(u)\,dy$. Choose $H_0=-\Lambda x^4c^2(u)$. In order to get
rid of the form~$A$, we solve the system of equations
$$
\frac{dx(u)}{du}=0,\qquad \frac{dy(u)}{du}=2\Lambda c(u) x^3(u)
$$
with the initial dates $x(0)=\widetilde x$ and $y(0)=\widetilde
y$. We get the transformation
$$
v=\widetilde v,\quad x=\widetilde x, \quad y=\widetilde y+2\Lambda
b(u)\widetilde x^{\,3},\quad u=\widetilde u,
$$
where the function~$b(u)$ satisfies $db(u)/du=c(u)$ and $b(0)=0$.
With respect to the new coordinates it holds
\begin{align*}
g&=2\,dv\,du+h(u)+\bigl(\Lambda v^2+3\Lambda
x^4c^2(u)\bigr)\,(du)^2,
\\
h(u)&=\frac{1}{-\Lambda\cdot
x^2}\bigl(\bigl(36\Lambda^2b^2(u)x^4+1\bigr) \,(dx)^2+12\Lambda
b(u)x^2\,dx\,dy+(dy)^2\bigr).
\end{align*}
Let $c(u)\equiv 1$, then $b(u)=u$ and $\det T=-9\Lambda^4
x^4(x^4+v^2)$. The equality ${\det T_m=0}$ ($m=(v,x,y,u)$) is
equivalent to the equality $v=0$. The metric~$g$ is
indecomposable. This metric is of Petrov type~D on the set
$\{(0,x,y,u)\}$ and of type~II on its complement.
\end{example}

\begin{example}
\label{ex3} The function $f=\ln\bigl(\tan(x/2)\bigr)\cos x+1$ is a
partial solution of the first equation in~\eqref{eq8.12}. We get
$A=\bigl(\cos x-\ln(\cot(x/2))\sin^2x\bigr)\,dy$. Consider the
transformation
$$
\widetilde v=v,\quad \widetilde x=x,\quad \widetilde y=y -\Lambda
u\biggl(\ln\biggl(\tan\frac{x}{2}\biggr)- \frac{\cos
x}{\sin^2x}\biggr),\quad \widetilde u=u.
$$
With respect to the new coordinates we have
\begin{gather*}
g=2\,dv\,du+h(u)+(\Lambda v^2+\widetilde H_0)\,(du)^2,
\\
h(u)=\biggl(\frac{1}{\Lambda}+\frac{4\Lambda
u^2}{\sin^4x}\biggr)\,(dx)^2 +\frac{4u}{\sin
x}\,dx\,dy+\frac{\sin^2x}{\Lambda}\,(dy)^2,
\end{gather*}
where~$\widetilde H_0$ satisfies the equation $\Delta_h\widetilde
H_0=-\dfrac{1}{2}h^{ij}\ddot h_{ij}$. An example of such a
function~$\widetilde H_0$ is
$$
\widetilde H_0=-\Lambda\biggl(\frac{1}{\sin^2x}+
\ln^2\biggl(\cot\frac{x}{2}\biggr)\biggr).
$$
It holds
$$
\det T=-\frac{\Lambda^4}{\sin^4x}\biggl(v^2+
\biggl(\ln\biggl(\cot\frac{x}{2}\biggr)\cos x-1\biggr)^2\biggr).
$$
Hence the metric~$g$ is of Petrov type~D on the set
$$
\biggl\{(0,x,y,u)\biggm|\ln\biggl(\cot\frac{x}{2}\biggr)\cos
x-1=0\biggr\}
$$
and of type~II on the complement to this set. The metric is
indecomposable.
\end{example}

\section{Riemannian and  Lorentzian manifolds with recurrent spinor fields}
\label{sec9}

Let $(M,g)$ be a pseudo-Riemannian spin manifold of signature
$(r,s)$, and $S$ the corresponding complex spinor bundle with the
induced connection $\nabla^S$. A spinor field $s\in \Gamma(S)$ is
called \textit{recurrent} if
\begin{equation}
\label{eq9.1} \nabla^S_Xs=\theta(X)s
\end{equation}
for all vector fields $X\in \Gamma(TM)$ (here $\theta$ is a
complex-valued 1-form). If $\theta=0$, then $s$ is \textit{a
parallel} spinor field. For a recurrent spinor field $s$ there
exists a locally defined non-vanishing function $f$ such that the
field~$fs$ is parallel  if and only if $d\theta=0$. If the
manifold $M$ is simply connected, then such function is defined
globally.

The study of Riemannian spin manifolds carrying parallel spinor
fields was initiated by Hitchin~\cite{83}, and then it was
continued by Friedrich \cite{54}. Wang characterized simply
connected Riemannian spin manifolds admitting parallel spinor
field in terms of their holonomy groups \cite{121}. A similar
result was obtained by Leistner for  Lorentzian
manifolds~\cite{98},~\cite{99},  by Baum and Kath for
pseudo-Riemannian manifolds with irreducible holonomy groups
\cite{15}, and by Ikemakhen in the case of pseudo-Riemannian
manifolds of neutral signature $(n,n)$ admitting two complementary
parallel isotropic distributions \cite{86}.

Friedrich~\cite{Fr} considered Equation~\eqref{eq9.1} on a
Riemannian spin manifold assuming that $\theta$ is a real-valued
1-form. He proved that this equation implies  that the Ricci
tensor is zero and $d\theta=0$. Below we will see that this
statement does not hold for Lorentzian manifolds. Example~1
from~\cite{54} provides a solution~$s$ to Equation \eqref{eq9.1}
with $\theta= i\omega$, $d\omega\ne 0$ for a real-valued 1-form
$\omega$ on the compact Riemannian manifold $(M,g)$ being the
product of the non-flat torus~$T^2$ and the circle $S^1$. In fact,
the recurrent spinor field~$s$ comes from a locally defined
recurrent spinor field on the non-Ricci-flat K\"ahler manifold
$T^2$; the existence of the last spinor field shows the below
given Theorem~\ref{th27}.

The spinor bundle~$S$ of a pseudo-Riemannian manifold $(M,g)$
admits a parallel one-dimensional complex subbundle if and only if
$(M,g)$ admits non-vanishing recurrent spinor fields in a
neighborhood of each point such that these fields are proportional
on the intersections of the domains of their definitions. In the
present section we study some classes of pseudo-Riemannian
manifolds $(M,g)$ whose spinor bundles admit parallel
one-dimensional complex subbundles.

\subsection{Riemannian manifolds}
\label{ssec9.1} Wang~\cite{121} showed that a simply connected
locally indecomposable  Riemannian manifold $(M,g)$ admits a
parallel spinor field if and only if its holonomy algebra
$\mathfrak{h}\subset\mathfrak{so}(n)$ is one of
$\mathfrak{su}(n/2)$, $\mathfrak{sp}(n/4)$, $G_2$,
$\mathfrak{spin}(7)$.

In~\cite{67} the following results for Riemannian manifolds with
recurrent spinor fields are obtained.

\begin{theorem}
\label{th27} Let $(M,g)$ be a locally indecomposable
$n$-dimensional simply connected
 Riemannian spin manifold. Then its spinor bundle $S$ admits a parallel one-dimensional complex subbundle
 if and only if either
the holonomy algebra $\mathfrak{h}\subset\mathfrak{so}(n)$ of the
manifold $(M,g)$ is one of $\mathfrak{u}(n/2)$,
$\mathfrak{su}(n/2)$, $\mathfrak{sp}(n/4)$,
$G_2\subset\mathfrak{so}(7)$,
$\mathfrak{spin}(7)\subset\mathfrak{so}(8)$, or $(M,g)$  is a
locally
 symmetric K\"ahlerian manifold.
\end{theorem}

\begin{corollary}
\label{cor4} Let $(M,g)$ be a simply connected  Riemannian spin
manifold with irreducible holonomy algebra and without non-zero
parallel spinor fields. Then the spinor bundle $S$ admits a
parallel one-dimensional complex subbundle if and only if $(M,g)$
is a  K\"ahlerian  manifold and it is not Ricci-flat.
\end{corollary}

\begin{corollary}
\label{cor5} Let $(M,g)$ be a simply connected complete
 Riemannian spin manifold without non-zero parallel spinor
  fields and with not irreducible holonomy algebra.
  Then its spinor bundle $S$ admits a parallel one-dimensional complex subbundle
if and only if $(M,g)$ is a direct product of a K\"ahlerian not
Ricci-flat spin manifold and of a Riemannian spin manifold with a
non-zero parallel spinor field.
\end{corollary}

\begin{theorem}
\label{th28} Let $(M,g)$ be a locally  indecomposable
$n$-dimensional simply connected non-Ricci-flat K\"ahlerian spin
manifold. Then its spinor bundle $S$ admits exactly two parallel
one-dimensional complex subbundles.
\end{theorem}

\subsection{Lorentzian manifolds}
\label{ssec9.2} The holonomy algebras of Lorentzian spin manifolds
admitting non-zero parallel spinor fields are classified
in~\cite{98},~\cite{99}. We suppose now that the spinor bundle of
$(M,g)$ admits a parallel one-dimensional complex subbundle and
$(M,g)$ does not admit any parallel spinor.

\begin{theorem}
\label{th29} Let $(M,g)$ be a simply connected complete Lorentzian
spin manifold. Suppose that $(M,g)$ does not admit a parallel
spinor. In this case the spinor bundle $S$ admits a parallel
one-dimensional complex subbundle if and only if one of the
following conditions holds:

{\rm1)}  $(M,g)$ is a direct product of $(\mathbb{R},-(dt)^2)$ and
of a Riemannian spin manifold $(N,h)$ such that the spinor bundle
of $(N,h)$  admits a parallel one-dimensional complex subbundle
and $(N,h)$ does not admit any non-zero parallel spinor field;

{\rm2)}  $(M,g)$ is a direct product of an indecomposable
Lorentzian spin manifold and of Riemannian spin manifold $(N,h)$
such that the spinor bundles of both manifolds admit parallel
one-dimensional complex subbundles and at least one of these
manifolds does not admit any non-zero parallel spinor field.
\end{theorem}

Consider locally indecomposable Lorentzian manifolds $(M,g)$.
Suppose that the spinor bundle of the manifold $(M,g)$  admits a
parallel one-dimensional complex subbundle~$l$. Let  $s\in
\Gamma(l)$ be a local non-vanishing section of the bundle~$l$.
 Let $p\in\Gamma(TM)$ be its Dirac current. The vector field~$p$ is defined from the equality

$$
g(p,X)=-\langle X \cdot s,s\rangle,
$$
where $\langle\,\cdot\,{,}\,\cdot\,\rangle$ is a Hermitian product
on $S$. It turns out that~$p$ is a recurrent vector field. In the
case of Lorentzian manifolds, the Dirac current satisfies
$g(p,p)\leqslant 0$ and the zeros of $p$ coincide with the zeros
of the field~$s$. Since $s$ is non-vanishing and $p$ is a
recurrent field, then either $g(p,p)<0$, or $g(p,p)=0$. In the
first case the manifold is decomposable. Thus we get that $p$ is
an isotropic recurrent vector field, and the manifold $(M,g)$
admits a parallel distribution of isotropic lines, i.e., its
holonomy algebra is contained in~$\mathfrak{sim}(n)$.

In~\cite{98},~\cite{99} it is shown that $(M,g)$ admits a parallel
spinor field if and only if
$\mathfrak{g}=\mathfrak{g}^{2,\mathfrak{h}}=\mathfrak{h}\ltimes\mathbb{R}^n$
and in the decomposition~\eqref{eq4.5} for the subalgebra
$\mathfrak{h}\subset\mathfrak{so}(n)$, each of the subalgebras
$\mathfrak{h}_i\subset\mathfrak{so}(n_i)$ coincides with one of
the Lie algebras $\mathfrak{su}(n_i/2)$, $\mathfrak{sp}(n_i/4)$,
$G_2\subset\mathfrak{so}(7)$,
$\mathfrak{spin}(7)\subset\mathfrak{so}(8)$.

В~\cite{67} we prove the following theorem.

\begin{theorem}
\label{Threc5} Let $(M,g)$ be a simply connected locally
indecomposable $(n+2)$-dimensional Lorentzian spin manifold. Then
its spinor bundle $S$ admits a parallel 1-dimensional complex
subbundle if and only if $(M,g)$ admits a parallel distribution of
isotropic lines (i.e., its holonomy algebra $\mathfrak{g}$ is
contained in $\mathfrak{sim}(n)$), and  in the
decomposition~\eqref{eq4.5} for the subalgebra
$\mathfrak{h}=\operatorname{pr}_{\mathfrak{so}(n)}\mathfrak{g}$
each of the subalgebra $\mathfrak{h}_i\subset\mathfrak{so}(n_i)$
coincedes with one of the Lie  algebras $\mathfrak{u}(n_i/2)$,
$\mathfrak{su}(n_i/2)$, $\mathfrak{sp}(n_i/4)$, $G_2$,
$\mathfrak{spin}(7)$ or with the holonomy algebra of an
indecomposable K\"ahlerian symmetric space. The number of parallel
1-dimensional complex subbundles of $S$ equals to the number of
$1$-dimensional complex subspaces of $\Delta_n$ preserved by the
algebra $\mathfrak{h}$.
\end{theorem}

\section{Conformally flat Lorentzian manifolds with special
holonomy groups} \label{sec10}

In this section will be given a local classification of
conformally flat Lorentzian manifolds with special holonomy
groups. The corresponding local metrics  are certain extensions of
Riemannian spaces of constant sectional curvature to Walker
metrics. This result is published in~\cite{64},~\cite{66}.

 Kurita~\cite{96} proved that
a conformally flat Riemannian manifold  is either a product of two
spaces of constant sectional curvature, or it is a product of a
space of constant sectional curvature with an interval, or its
restricted holonomy group is the identity component of the
orthogonal group. The last condition represents the generic  case,
and among various manifolds satisfying the last condition one can
emphasize only the spaces of constant sectional curvature. It is
clear that there are no conformally flat Riemannian manifolds with
special holonomy groups.

 In~\cite{66} we generalize the Kurita Theorem to the case of
pseudo-Riemannian manifolds.  It turns out that in additional to
the above listed possibilities a conformally flat
pseudo-Riemannian manifold may have weakly irreducible not
irreducible holonomy group. We give a complete local description
of conformally flat  Lorentzian manifolds $(M,g)$ with weakly
irreducible not irreducible holonomy groups.

On a Walker manifold $(M,g)$  we define the canonical function
$\lambda$ from the equality
$$
\operatorname{Ric}(p)=\lambda p,
$$
where $\operatorname{Ric}$ is the Ricci operator. If the
metric~$g$ is written in the form~\eqref{eq4.11}, then
$\lambda=(1/2)\partial^2_vH$, and the scalar curvature of the
metric~$g$ satisfies
$$
s=2\lambda+s_0,
$$
where $s_0$ is the scalar curvature of the metric~$h$. The form of
a conformally flat Walker matric will depend on the vanishing of
the function~$\lambda$. In the general case we obtain the
following result.

\begin{theorem}
\label{th31} Let $(M,g)$ be a conformally flat Walker manifold
(i.e., the Weyl  curvature tensor equals to zero) of dimension
$n+2\geqslant 4$. Then in a neighborhood of each point of~$M$
there exist coordinates $v,x^1,\dots,x^n,u$ such that
$$
g=2\,dv\,du+\Psi\sum_{i=1}^n(dx^i)^2+2A\,du+
(\lambda(u)v^2+vH_1+H_0)\,(du)^2,
$$
where
\begin{gather*}
\Psi=4\biggl(1-\lambda(u)\sum_{k=1}^n(x^k)^2\biggr)^{-2},
\\
A=A_i\,dx^i,\quad
A_i=\Psi\biggl(-4C_k(u)x^kx^i+2C_i(u)\sum_{k=1}^n(x^k)^2\biggr),
\\
H_1=-4C_k(u)x^k\sqrt{\Psi}-\partial_u\ln\Psi+K(u),
\\
H_0(x^1,\dots,x^n,u)=\hspace*{90mm} \\
=\begin{cases}
\dfrac{4}{\lambda^2(u)}\Psi\displaystyle\sum_{k=1}^nC^2_k(u)&
\\
\qquad{}+\sqrt{\Psi}\,\biggl(a(u)\displaystyle\sum_{k=1}^n(x^k)^2
+D_k(u)x^k+D(u)\biggr),&\text{if} \ \lambda(u)\ne 0,
\\[4mm]
16\biggl(\,\displaystyle\sum_{k=1}^n(x^k)^2\biggr)^2
\displaystyle\sum_{k=1}^nC^2_k(u)&
\\
\qquad{}+\widetilde a(u)\displaystyle\sum_{k=1}^n(x^k)^2+
\widetilde D_k(u)x^k+\widetilde D(u),& \text{if} \ \lambda(u)=0,
\end{cases}
\end{gather*}
for some functions  $\lambda(u)$, $a(u)$, $\widetilde a(u)$,
$C_i(u)$, $D_i(u)$, $D(u)$, $\widetilde D_i(u)$, $\widetilde
D(u)$.

The scalar curvature of the metric~$g$ is equal to
$-(n-2)(n+1)\lambda(u)$.
\end{theorem}

If the function~$\lambda$ is locally zero, or it is non-vanishing,
then the above metric may be simplified.

\begin{theorem}
\label{th32} Let $(M,g)$ be a conformally flat Walker Lorentzian
manifold of dimension $n+2\geqslant 4$.

{\rm1)} If the function $\lambda$ is non-vanishing at a point,
then in a neighborhood of this point there exist coordinates
$v,x^1,\dots,x^n,u$ such that
$$
g=2\,dv\,du+\Psi\sum_{i=1}^n(dx^i)^2+(\lambda(u)v^2+vH_1+H_0)\,(du)^2,
$$
where
\begin{gather*}
\Psi=4\biggl(1-\lambda(u)\sum_{k=1}^n(x^k)^2\biggr)^{-2},
\\*[3mm]
H_1=-\partial_u\ln\Psi,\qquad
H_0=\sqrt{\Psi}\,\biggl(a(u)\sum_{k=1}^n(x^k)^2+
D_k(u)x^k+D(u)\biggr).
\end{gather*}

{\rm2)} If $\lambda\equiv 0$ in a neighborhood of a point, then in
a neighborhood of this point there exist coordinates
$v,x^1,\dots,x^n,u$ such that
$$
g=2\,dv\,du+\sum_{i=1}^n(dx^i)^2+2A\,du+(vH_1+H_0)\,(du)^2,
$$
where
\begin{gather*}
A=A_i\,dx^i,\quad A_i=C_i(u)\sum_{k=1}^n(x^k)^2,\quad
H_1=-2C_k(u)x^k,
\\
{\begin{align*}
H_0&=\sum_{k=1}^n(x^k)^2\biggl(\frac{1}{4}\sum_{k=1}^n(x^k)^2
\sum_{k=1}^nC^2_k(u)-(C_k(u)x^k)^2+\dot C_k(u)x^k+a(u)\biggr)
\\
&\qquad+D_k(u)x^k+D(u).
\end{align*}}
\end{gather*}
In particular, if all~$C_i$ equal~$0$, then the metric can be
rewritten in the form
\begin{equation}
\label{eq10.1}
g=2\,dv\,du+\sum_{i=1}^n(dx^i)^2+a(u)\sum_{k=1}^n(x^k)^2\,(du)^2.
\end{equation}
\end{theorem}

Thus Theorem \ref{th32} gives the local form of a conformally flat
Walker metric in the neighborhoods of  points where~$\lambda$ is
non-zero or constantly zero. Such points represent a dense subset
of the manifold. Theorem \ref{th31} describes also the metric in
the neighborhoods of points at that the function $\lambda$
vanishes, but it is not locally zero, i.g. in the neighborhoods of
isolated zero points of~$\lambda$.

Next, we find the holonomy algebras of the obtained metrics and
check which of the  metrics are decomposable.

\begin{theorem}
\label{th33} Let $(M,g)$ be as in Theorem~{\rm\ref{th31}}.

{\rm1)} The manifold $(M,g)$ is locally indecomposable if and only
if there exists a coordinate system with one of the properties:
\begin{enumerate}
\item[$\bullet$] $\dot\lambda\not\equiv 0$;
\item[$\bullet$] $\dot\lambda\equiv 0$, $\lambda\ne 0$,
i.e.,~$g$ can be written as in the first part of
Theorem~{\rm\ref{th32}}, and
$$\displaystyle\sum_{k=1}^nD^2_k+(a+\lambda D)^2\not\equiv 0;$$
\item[$\bullet$] $\lambda\equiv 0$, i.e.,~$g$ can be written as in the second part of
Theorem~{\rm\ref{th32}},~and
$$
\displaystyle\sum_{k=1}^nC^2_k+a^2\not\equiv 0.
$$
\end{enumerate}

Otherwise, the metric can be written in the form
$$
g=\Psi\sum_{k=1}^n(dx^k)^2+2\,dv\,du+\lambda v^2\,(du)^2,\qquad
\lambda\in\mathbb{R}.
$$
The holonomy algebra of this metric is trivial if and only if
$\lambda=0$. If $\lambda\ne 0$, then the holonomy algebra is
isomorphic to $\mathfrak{so}(n)\oplus\mathfrak{so}(1,1)$.

{\rm2)} Suppose that the manifold~$(M,g)$ is  locally
indecomposable. Then its holonomy algebra is isomorphic to
$\mathbb{R}^n\subset\mathfrak{sim}(n)$ if and only if
$$\lambda^2+\displaystyle\sum_{k=1}^nC^2_k\equiv0$$ for all coordinate systems.
In this case $(M,g)$ is a pp-wave, and $g$ is given
by~\eqref{eq10.1}. If for each coordinate system it holds
$$\lambda^2+\displaystyle\sum_{k=1}^nC^2_k\not\equiv 0,$$
then the  holonomy algebra is isomorphic to $\mathfrak{sim}(n)$.
\end{theorem}

Possible holonomy algebras of conformally flat 4-dimensional
Lorentzian manifolds are classified  in \cite{81}, in this paper
it was posed the problem to construct an example of
 conformally flat metric with the holonomy algebra
$\mathfrak{sim}(2)$ (which is denoted in \cite{81} by $R_{14}$).
An attempt to construct such metric was done in \cite{75}. We show
that the metric constructed there is in fact decomposable and its
holonomy algebra is $\mathfrak{so}(1,1)\oplus\mathfrak{so}(2)$.
 Thus in this paper we get conformally flat metrics with the
holonomy algebra $\mathfrak{sim}(n)$ for the first time, and even
more,  we find all such metrics.

The field equations of Nordstr\"om's theory of gravitation, which
was originated before Einstein's theory have the form
$$W=0,\qquad s=0$$
(see~\cite{106},~\cite{119}). All metrics from Theorem \ref{th31}
in dimension 4 and metrics from the second part of Theorem
\ref{th32} in bigger dimensions
 provide examples of solutions of these equations. Thus we have found all solutions to
Nordstr\"om's gravity with holonomy algebras contained in
$\mathfrak{sim}(n)$. Above we have seen that  it is impossible to
obtain the complete solution of the Einstein equation on
Lorentzianы manifolds with such holonomy algebras.

An important fact is that a simply connected conformally flat spin
Lorentzian manifold  admits the space of conformal Killing spinors
of maximal dimension~\cite{12}.

It would be interesting to obtain examples of conformally flat
Lorentzian manifolds satisfying some global geometric properties,
e.g., important are globally hyperbolic Lorentzian manifolds with
special holonomy groups~\cite{17},~\cite{19}.

The projective equivalence of 4-dimensional conformally flat
Lorentzian metrics with special holonomy algebras was studied
recently in~\cite{80}. There are many interesting works about
conformally flat (pseudo-)Riemannian, and in particular Lorentzian
manifolds. Let us mention some of them:~\cite{6},~\cite{84},
\cite{93},~\cite{115}.

\section{2-symmetric Lorentzian manifolds}
\label{sec11}

In this section we discuss the classification of 2-symmetric
 Lorentzian manifolds obtained in~\cite{5}.

Symmetric pseudo-Riemannian manifolds constitute  an important
class of spaces. A direct generalization of these manifolds is
provided by the so-called $k$-symmetric pseudo-Riemannian spaces
$(M,g)$ satisfying the conditions
$$
\nabla^k R=0,\quad \nabla^{k-1} R\ne 0,
$$
where $k\geqslant 1$. In the case of Riemannian manifolds, the
condition $\nabla^k R=0$ implies $\nabla R=\nobreak0$~\cite{117}.
On the other hand, there exist pseudo-Riemannian $k$-symmetric
spaces for $k\geqslant 2$~\cite{28},~\cite{90},~\cite{112}.

Indecomposable  simply connected Lorentzian symmetric spaces are
exhausted by the de~Sitter, the anti-de~Sitter spaces and by the
Cahen-Wallach spaces, which are special pp-waves.
Kaigorodov~\cite{90} considered different generalizations of
Lorentzian symmetric spaces.

The paper by Senovilla~\cite{112} starts systematic investigation
of 2-symmetric Lorentzian spaces. In this paper it is proven that
any 2-symmetric Lorentzian space admits a parallel isotropic
vector field. In the paper~\cite{28} a classification of
four-dimensional 2-symmetric Lorentzian spaces is obtained, for
that the Petrov classification of the Weyl tensors~\cite{108} was
used.

In~\cite{5} we generalize the result~\cite{28} to the case of
arbitrary dimension.

\begin{theorem}
\label{th34} Let $(M,g)$ be a locally indecomposable Lorentzian
manifold of dimension $n+2$. Then $(M,g)$ is 2-symmetric if and
only if locally there exist coordinates $v,x^1,\dots,x^n,u$ such
that
$$
g=2\,dv\,du+\sum_{i=1}^n(dx^i)^2+(H_{ij}u+F_{ij})x^ix^j\,(du)^2,
$$
where $H_{ij}$ is a nonzero diagonal real matrix with the diagonal
elements $\lambda_1\leqslant\cdots\leqslant\lambda_n$, а $F_{ij}$
is a symmetric real matrix.
\end{theorem}

From  the Wu Theorem  it follows that any 2-symmetric Lorentzian
manifold  is locally a product of an indecomposable  2-symmetric
Lorentzian manifold and of  a locally symmetric Riemannian
manifold. In~\cite{29} it is shown that a simply connected
geodesically complete  2-symmetric Lorentzian manifold is the
product of~$\mathbb{R}^{n+2}$ with the metric from
Theorem~\ref{th34} and of (possibly trivial) Riemannian symmetric
space.

The proof of Theorem~\ref{th34} given in~\cite{5}, demonstrates
the methods of the theory of the holonomy groups in the best way.
Let $\mathfrak{g}\subset\mathfrak{so}(1,n+1)$ be the holonomy
algebra of the manifold $(M,g)$. Consider the
space~$\mathscr{R}^\nabla(\mathfrak{g})$ of covariant derivatives
of the algebraic curvature tensors of type~$\mathfrak{g}$,
consisting of the linear maps from~$\mathbb{R}^{1,n+1}$
to~$\mathscr{R}(\mathfrak{g})$ that satisfy the second Bianchi
identity. Let
$\mathscr{R}^\nabla(\mathfrak{g})_\mathfrak{g}\subset
\mathscr{R}^\nabla(\mathfrak{g})$ be the subspace annihilated by
the algebra~$\mathfrak{g}$.

The tensor~$\nabla R$ is  parallel and non-zero, hence its value
at each point of the  manifold belongs to the  space
$\mathscr{R}^\nabla(\mathfrak{g})_\mathfrak{g}$. The space
$\mathscr{R}^\nabla(\mathfrak{so}(1,n+1))_\mathfrak{g}$ is
trivial~\cite{116}, therefore
$\mathfrak{g}\subset\mathfrak{sim}(n)$.

The corner stone of the proof is the equality
$\mathfrak{g}=\mathbb{R}^n\subset\mathfrak{sim}(n)$, i.e.,
$\mathfrak{g}$ is the algebra of type~ 2 with trivial orthogonal
part~$\mathfrak{h}$. Such manifold is a pp-wave (see
Section~\ref{ssec4.2}), i.e., locally it holds
$$
g=2\,dv\,du+\sum_{i=1}^n(dx^i)^2+H\,(du)^2,\qquad
\partial_v H=0,
$$
and the equation $\nabla^2 R=0$ can be easily solved.

Suppose that the orthogonal part
$\mathfrak{h}\subset\mathfrak{so}(n)$ of the holonomy
algebra~$\mathfrak{g}$  is non-trivial. The subalgebra
$\mathfrak{h}\subset\mathfrak{so}(n)$ can be decomposed into
irreducible parts, as in Section~\ref{ssec4.1}. Using the
coordinates~\eqref{eq4.12} allows to assume that the  subalgebra
$\mathfrak{h}\subset\mathfrak{so}(n)$ is irreducible.
If~$\mathfrak{g}$ is of type~1 or~3, then simple algebraic
computations show that
$\mathscr{R}^\nabla(\mathfrak{g})_\mathfrak{g}=0$.

We are left with the case
$\mathfrak{g}=\mathfrak{h}\ltimes\mathbb{R}^n$, where the
subalgebra $\mathfrak{h}\subset\mathfrak{so}(n)$ is irreducible.
In this case the space
$\mathscr{R}^\nabla(\mathfrak{g})_\mathfrak{g}$ is
one-dimensional, which allows us to find the explicit form of the
tensor~$\nabla R$, namely, if the metric~$g$ has the
from~\eqref{eq4.11}, then
$$
\nabla R=f\,du\otimes h^{ij} (p\wedge \partial_i)\otimes (p\wedge
\partial_j)
$$
for some function~$f$.

Next,  using the last equality it was proved that  $\nabla W=0$,
i.e., the Weyl conformal tensor $W$ is parallel. The results of
the paper~\cite{49}  show that either $\nabla R=0$, or $W=0$, or
the manifold under the consideration is a pp-wave. The first
condition contradicts the assumption $\nabla R\ne 0$, the last
condition contradicts the assumption  $\mathfrak{h}\ne 0$. From
the results of Section~\ref{sec10} it follows that the condition
 $W=0$ implies the equality $\mathfrak{h}=0$,
i.e., we again get a contradiction.

It turns out that the last step of the proof from~\cite{5} can be
appreciably simplified and it is not necessary to consider the condition $\nabla
W=0$. Indeed, let us turn back to the equality for~$\nabla R$. It
is easy to check that $\nabla du=0$. Therefore the equality
 $\nabla^2R=0$ implies $\nabla(f h^{ij} (p\wedge
\partial_i)\otimes (p\wedge
\partial_j))=0$. Consequently,
$$
\nabla (R-uf h^{ij} (p\wedge \partial_i)\otimes (p\wedge
\partial_j))=0.
$$
The value of the tensor field $R-uf h^{ij}(p\wedge
\partial_i)\otimes (p\wedge \partial_j)$ at each point of the manifold
belongs to the space~$\mathscr{R}(\mathfrak{g})$ and it is
annihilated by the holonomy algebra~$\mathfrak{g}$. This
immediately implies that
$$
R-uf h^{ij}(p\wedge\partial_i)\otimes (p\wedge \partial_j)= f_0
h^{ij}(p\wedge \partial_i)\otimes(p\wedge \partial_j)
$$
for some function~$f_0$, i.e., $R$ is the curvature tensor of a
pp-wave, which contradicts the condition $\mathfrak{h}\ne 0$.
Thus, $\mathfrak{h}=0$, and
$\mathfrak{g}=\mathbb{R}^n\subset\mathfrak{sim}(n)$.

Theorem~\ref{th34} was reproved in~\cite{29} by another method.

\end{fulltext}


\end{document}